\documentclass[11pt]{article}
\usepackage[mathlines]{lineno}
\usepackage[margin=1.2in]{geometry} % 设置页边距为1.2英寸
\usepackage{indentfirst} % 设置每个段落首行缩进
\usepackage{amsmath}      % 数学公式
\usepackage{amsthm}       % 定理环境
\usepackage{amssymb}      % 数学符号
\usepackage{algorithm}    % 算法伪代码
\usepackage{algpseudocode}
\usepackage{graphicx}     % 插入图片
\usepackage{booktabs}     % 表格格式
\usepackage{cite}         % 参考文献管理
\usepackage{fancyhdr}     % 页脚作者信息

\usepackage{amsmath}
\usepackage{enumitem}
\usepackage{listings}
\usepackage{xcolor}
\usepackage{natbib}

\usepackage{multirow}

\usepackage{float} % 加载宏包
\newcommand{\argmin}{\mathop{\arg\min}}

\newtheorem{theorem}{Theorem}[section]
\newtheorem{lemma}{Lemma}[section]
\newtheorem{proposition}{Proposition}[section]
\newtheorem{definition}{Definition}[section]
\newtheorem{assumption}{Assumption}[section] 

\newtheorem{corollary}{Corollary}[section]
\numberwithin{equation}{section} % 公式编号包含章节号

\title{An efficient proximal algorithm for squared L1 over L2 regularized sparse recovery\thanks{		
		This work was supported by the National Natural Science Foundation of China under grants 12271181 and 12471098, and by the Guangzhou Basic Research Program under grant 2025A04J5240, and by the Basic and Applied Basic Research Foundation of Guangdong Province under grant 2023A1515030046, and by the Guangdong Province Key Laboratory of Computational Science at the Sun Yat-sen University (2020B1212060032).
		\newline $\indent\: ^{1}$Department of Applied Mathematics, College of Mathematics and Informatics, South China Agricultural University, Guangzhou 510642, China. 
		\newline $\indent\: ^{2\dagger}$Corresponding author. School of Computer Science and Engineering, Guangdong Province Key Laboratory of Computational Science, Sun Yat-sen University, Guangzhou 510275, China (liqia@mail.sysu.edu.cn).
		\newline $\indent\: ^{3}$School of Computer Science and Engineering, Sun Yat-sen University, Guangzhou 510275, China.
	}
}

\author{Na Zhang$^{1}$
\and Hong Chen$^{1}$
\and Qia Li$^{2\dagger}$
\and Junpeng Zhou$^{3}$}
\begin{document}
\maketitle	
% \linenumbers
% REQUIRED
\begin{abstract}
In this paper, we consider a squared $L_1/L_2$ regularized model for sparse signal recovery from noisy measurements. We first establish the existence of optimal solutions to the model under mild conditions. Next, we propose a proximal method for solving a general fractional optimization problem which has the squared $L_1/L_2$ regularized model as a special case. We prove that any accumulation point of the solution sequence generated by the proposed method is a critical point of the fractional optimization problem. Under additional KL assumptions on some potential function, we establish the sequential convergence of the proposed method. When this method is specialized to the squared $L_1/L_2$ regularized model, the proximal operator involved in each iteration admits a simple closed form solution that can be computed with very low computational cost. Furthermore, for each of the three concrete models, the solution sequence generated by this specialized algorithm converges to a critical point. Numerical experiments demonstrate the superiority of the proposed algorithm for sparse recovery based on squared $L_1/L_2$ regularization.
\end{abstract}
	
% REQUIRED
	
	% REQUIRED
	%\begin{MSCcodes}
	%90C26, 90C30, 90C90
	%\end{MSCcodes}
	
\section{Introduction}
    The compressive sensing (CS) problem seeks to estimate the unknown sparse signal $x \in \mathbb{R}^n$ from the possibly noisy measurement $b\in \mathbb{R}^m$ $(m\ll n)$ given by 
    \[
    b = Ax + z,
    \]
    where $A\in \mathbb{R}^{m\times n}$ is a sensing matrix and $z\in \mathbb{R}^m$ is some noise vector. Mathematically, in the case where the noise $z$ follows Gaussian distribution, a typical regularization method for CS takes the form of 
    \begin{equation}\label{equation1.1}
    	\underset{x \in \mathbb{R}^n}{\min} \lambda R(x)+\frac{1}{2}\|Ax-b\|_2^2,
    \end{equation}
    where $R(\cdot)$ is the sparsity regularizer and $\lambda>0$ is a regularization parameter balancing the two terms involved. Various choices of the sparsity regularizer have been introduced and successfully used in \eqref{equation1.1}, such as the $\ell_0$ function \cite{BLUMENSATH2009265}, $\ell_{1}$ norm \cite{doi:10.1137/S1064827596304010}, the $\ell_p$ quasi norm $(0<p<1)$ \cite{4303060,6205396}, the log-sum penalty function \cite{candes2008enhancing}, the minimax-concave penalty function \cite{Zhang2010NearlyUV}, the capped-$\ell_1$ function \cite{JMLR:v11:zhang10a}, the $\ell_1-\ell_2$ function \cite{2013A} and so on. Among these sparsity regularizers, the $\ell_0$ function $\|x\|_0$, which counts the number of nonzero entries in $x$, is the most natural mathematical measure of sparsity and plays an important theoretical role in many aspects of CS. However, it has a practical drawback of being highly sensitive to small entries due to its discontinuity, see \cite{DBLP:journals/corr/abs-1204-4227}, for example. All the other regularizers can be viewed as continuous approximations to the $\ell_0$ function, and designed to alleviate this drawback.\\ 
    \indent In \cite{5985551,DBLP:journals/corr/abs-1204-4227,7506083}, the squared $\ell_1/\ell_2$ norm, $\|\cdot\|_1^2/\|\cdot\|_2^2$ is highly recommended as an alternative measure of sparsity and the reasons are roughly summarized as follows. First, it provides a continuous sharp lower bound for $\|\cdot\|_0$, that is, for all nonzero $x$, it holds that
    \[
    \frac{\|x\|_1^2}{\|x\|_2^2}\leq \|x\|_0,
    \]
    where the equality is attained if and only if the nonzero entries of $x$ are equal in magnitude. Second, the squared $\ell_1/\ell_2$ norm can be interpreted as an important example of the so-called effective sparsity, which is defined based on entropy and quantifies the “effective number of coordinates of $x$”. Third, it is also a sensible measure of sparsity for non-idealized signal, as illustrated in \cite[Frgure 1]{DBLP:journals/corr/abs-1204-4227}. Despite its theoretical advantage, in existing works the squared $\ell_1/\ell_2$ norm has seldom been used for sparsity regularization in model \eqref{equation1.1} or other practical applications. In contrast, the $\ell_1/\ell_2$ norm, has been successfully used for sparsity regularization in many applications, such as CS \cite{doi:10.1137/18M123147X,2014Ratio,zeng2021analysis,doi:10.1137/20M136801X,2023Study,han2025single}, CT reconstruction \cite{doi:10.1137/20M1341490,wang2022minimizing}, blind deconvolution \cite{5995521}, nonnegative matrix factorization \cite{DBLP:journals/corr/cs-LG-0408058} and so on. Very recently, the authors consider the following constrained squared $\ell_1/\ell_2$ minimization problem for sparse signal recovery in \cite{jia2024sparserecoverysquareell1ell2}  
    \begin{equation}\label{equati:1.2}
    	\underset{x\in\mathbb{R}^n}{\min} \; \left\{\frac{\|x\|_1^2}{\|x\|_2^2} :\|Ax - b\|_2\leq \sigma\right\},
    \end{equation}
    where $0<\sigma<\|b\|_2$ measures the noise level. However, this model differs from \eqref{equation1.1} where $R$ is chosen as the squared $\ell_1/\ell_2$ norm. In fact, it can be viewed as a $\ell_1/\ell_2$ regularized model, because it is clearly equivalent to the formulation obtained by replacing its objective function with the $\ell_1/\ell_2$ norm. Given the above background, we are motivated to investigate the squared $\ell_2/\ell_2$ regularized model for CS, with emphasis on developing efficient numerical methods for solving the associated optimization problems.   \\   
    \indent In this paper, we consider the following squared $\ell_1/\ell_2$ regularized model for CS with noisy measurements 
    \begin{equation}
    	\underset{x\in\mathcal{X}}{\min} \; \left\{\lambda\frac{\|x\|_1^2}{\|x\|_2^2} + q(Ax - b)\right\},
    	\label{eq:problem1}
    \end{equation}
    where the loss function $q: \mathbb{R}^m\to[0,+\infty)$ can be written as $q=q_1-q_2$ with $q_1:\mathbb{R}^m \to \mathbb{R}$ being continuously differentiable with Lipschitz continuous gradient and $q_2:\mathbb{R}^m \to \mathbb{R}$ being convex continuous, and $\mathcal{X}\subseteq \mathbb{R}^n$ is closed hyperrectangle (possibly unbounded with $0\in\mathcal{X}$). Below we list three concrete examples of \eqref{eq:problem1}, which are equipped with some widely used loss functions for modeling different types of noise.
    \begin{itemize}
    	\item Gaussian noise: When the noise in the measurement follows the Gaussian distribution, the quadratic function $\frac{1}{2}\|\cdot\|_2^2$ is typically employed as the loss function $q$ \cite{candes2006stable,chen2001atomic}. One may consider the following problem as
    	\begin{equation}
    		\underset{x\in\mathcal{X}}{\min} \;\left\{\lambda\frac{\|x\|_1^2}{\|x\|_2^2}+\frac{1}{2}\|Ax-b\|_2^2\right\}.
    		\label{eq:1.2}
    	\end{equation}
    	Problem \eqref{eq:1.2} is a special instance of model \eqref{eq:problem1} with $q(y)=q_1(y)=\frac{1}{2}\|y\|_2^2$ and $q_2=0$.
    	\item Cauchy noise: When the noise of the measurement follows the Cauchy distribution, the Lorentzian norm  $\|y\|_{LL_2,\gamma}:=\sum_{i=1}^m \log(1+\gamma^{-2}y_{i}^2)$ with $\gamma>0$ is usually utilized to model the noise \cite{carrillo2010robust,carrillo2016robust}. In this scenario, one may consider the following problem:
    	\begin{equation}
    		\underset{x\in\mathcal{X}}{\min} \;\left\{\lambda\frac{\|x\|_1^2}{\|x\|_2^2}+\|Ax-b\|_{LL_2,\gamma}\right\}.
    		\label{eq:1.3}
    	\end{equation} 
    	Note that the Lorentzian norm is continuously differentiable with Lipschitz continuous gradient. Then problem  \eqref{eq:1.3} corresponds to \eqref{equation1.1} with $q(y)=q_1(y)=\|y\|_{LL_2,\gamma}$ and $q_2=0$.
    	\item Robust compressed sensing: In this scenario, the measurement is corrupted by both Gaussian noise and electromyographic noise \cite{carrillo2016robust,polania2012compressive}, where the latter noise is sparse and may have large outliers. As in \cite[Section 5.1]{liu2019refined}, one may adopt the loss function $q(y)=\frac{1}{2}\mathrm{dist}^2(y,S_r)$, where $S_r := \{z\in\mathbb{R}^m : \|z\|_0\leq r\}$ and $r$ is an estimate of the number of outliers. This yields the following problem
    	\begin{equation}
    	\underset{x\in\mathcal{X}}{\min}\left\{\lambda\frac{\|x\|_1^2}{\|x\|_2^2}+\frac{1}{2}\mathrm{dist}^2(Ax-b,S_r)\right\}.
    		\label{eq:1.4}
    	\end{equation} 
    	Notice that
    	\[
    	\begin{aligned}
    		\frac{1}{2}\mathrm{dist}^2(y,S_r)&=\frac{1}{2}\min\{\|y-z\|_2^2:\|z\|_0\leq r\}\\
    		&=\frac{1}{2}\|y-T_r(y)\|_2^2=\frac{1}{2}\|y\|_2^2-\frac{1}{2}\|T_r(y)\|_2^2,
    	\end{aligned}	
    	\]
    	where $T_r$ is the projection operator onto $S_r$. It is not hard to verify the convexity of $\|T_r(y)\|_2^2$. Hence, problem \eqref{eq:1.4} is a special instance of model \eqref{eq:problem1} with $q_1(y)=\frac{1}{2}\|y\|^2$ and $q_2(y)=\frac{1}{2}\|T_r(y)\|_2^2$.
    \end{itemize}  
    \hspace*{\parindent}To the best of our knowledge, existing first-order methods for general nonsmooth optimization are confined to solving \eqref{eq:problem1} in the case of $\mathcal{X}=\mathbb{R}^n$. This limitation arises because all these methods require computing the proximal operator of $\lambda\|\cdot\|_1^2/\|\cdot\|_2^2+\iota_{\mathcal{X}}$ at each iteration when they are applied to \eqref{eq:problem1}. Here, $\iota_{\mathcal{X}}$ denotes the indicator function of $\mathcal{X}$.    
    However, closed-form solutions for this proximal operator are only available in the case of $\mathcal{X}=\mathbb{R}^n$ \cite{jia2024computingproximityoperatorsscale}, and it is unknown for a general $\mathcal{X}$. Moreover, even when $\mathcal{X}=\mathbb{R}^n$, computing these closed-form solutions has high complexity, which may adversely affecting the computational efficiency. To overcome these challenges, in this work we develop an efficient proximal algorithm for solving model by exploiting its fractional structure. The proposed method is not only applicable to the model with an arbitrary closed hyperrectangle $\mathcal{X}$ but also enjoys low computational complexity per iteration. Specifically, the contributions of this paper are summarized as follows. First, we present some conditions that rigorously ensure the existence of optimal solutions to \eqref{eq:problem1}. Note that this issue is not trivial as the extended objective of \eqref{eq:problem1} is generally not closed and the constrained set $\mathcal{X}$ may be unbounded. In particular, we relate the existence of optimal solutions to the $s$-spherical section property \cite{vavasis2009derivation,zhang2013theory} of $\operatorname{ker} A$, when $\mathcal{X}$ is unbounded. Next, we propose a novel first-order method for solving a general fractional program which has model \eqref{eq:problem1} as a special case. Different from the aforementioned methods, the proposed method requires computing the proximal operator of $\sqrt{\lambda}\|\cdot\|_1+\iota_{\mathcal{X}}$ when it is applied to the model \eqref{eq:problem1}. This operator admits a simple closed form solution that can be evaluated with very low computational cost. We prove that any accumulation point of the solution sequence generated by the proposed algorithm is a critical point of the fractional program.  
    Moreover, by imposing additional KL assumptions on some potential function, we establish the sequential convergence of the proposed method. Consequently, we prove the sequential convergence of the proposed method to a critical point when it is specialized to solving models \eqref{eq:1.2}, \eqref{eq:1.3} and \eqref{eq:1.4}. Finally, we perform numerical experiments on sparse recovery from noisy measurements to demonstrate the performance of our proposed algorithm.\\
    \indent This paper is organized as follows. In section \ref{sec2}, we introduce notation and review some preliminaries. The existence of optimal solutions for \eqref{eq:problem1} is studies in section \ref{sec3}. In section \ref{sec4}, we propose our numerical algorithm and establish its subsequent convergence. The sequential convergence of the proposed algorithm is further analyzed in section \ref{sec6}. Finally, we present numerical results in section \ref{sec:6}.
    \section{Notation and preliminaries} \label{sec2}
    We begin with our preferred notations. We denote by $\mathbb{R}^n$ and $\mathbb{N}$ the $n$-dimensional Euclidean space and the set of nonnegative integers, respectively. The $\ell_{1}$-norm, the $\ell_{2}$-norm and the inner-product of $\mathbb{R}^n$ are denoted respectively by $\|\cdot\|_1$, $\|\cdot\|_2$ and $ \langle\cdot,\cdot\rangle$. We use  $\mathcal{B}(x,\delta)$ to denote an open ball centered at $x \in \mathbb{R}^n$ with radius $ \delta > 0 $, i.e., 
    $\mathcal{B}(x,\delta) := \{ z \in \mathbb{R}^{n} : \|x - z\|_2 < \delta \}$. The Cartesian product of the sets $\mathcal{S}_1$ and $ \mathcal{S}_2$ is denoted by $ \mathcal{S}_1 \times \mathcal{S}_2 $. The distance from a vector $x \in \mathbb{R}^{n}$ to a set $ \mathcal{S} \subseteq \mathbb{R}^{n}$ is denoted by 
    $\mathrm{dist}(x, \mathcal{S}) := \inf \{ \|x - y\|_2 : y \in \mathcal{S} \}$. The indicator function on a nonempty set $\mathcal{S} \subseteq \mathbb{R}^{n}$ is defined by
    \[
    \iota_{\mathcal{S}}(x) = 
    \begin{cases} 
    	0, & \text{if } x \in \mathcal{S}, \\ 
    	+\infty, & \text{else}. 
    \end{cases}
    \]
    \indent We next review some preliminaries on subdifferentials of nonconvex functions \cite{cui2021modern,Mordukhovich2006,rockafellar1998variational}. An extended-real-valued function $ \varphi : \mathbb{R}^{n} \to (-\infty, +\infty]$ is said to be proper if its domain $ \mathrm{dom} \varphi := \{ x \in \mathbb{R}^{n} : \varphi(x) < +\infty \}$ is nonempty. A proper function $ \varphi $ is said to be closed if it is lower semi-continuous on $\mathbb{R}^{n}$. For a proper function $\varphi$, the Fréchet subdifferential at $x \in \mathrm{dom}\varphi$ is defined by
    \[
    \widehat{\partial}\varphi(x) := \left\{ y \in \mathbb{R}^n : \liminf_{\substack{z \to x \\ z \neq x}} \frac{\varphi(z) - \varphi(x) - \langle y, z - x \rangle}{\| z - x \|_2} \geq 0 \right\},
    \]
    and the limiting subdifferential at $x \in \mathrm{dom}\varphi$ is defined by
    \[
    \partial \varphi(x) := \left\{ y \in \mathbb{R}^n : \exists x^k \to x,  \varphi(x^k) \to \varphi(x), y^k \to y \text{ with } y^k \in \widehat{\partial} \varphi(x^k) \text{ for each } k \right\}.
    \]
    We set $\widehat{\partial} \varphi(x) = \partial \varphi(x) = \emptyset$ by convention when $x \notin \mathrm{dom}\varphi$, and define $\mathrm{dom}\partial \varphi := \{ x : \partial \varphi(x) \neq \emptyset \}$. It holds that $\widehat{\partial} \varphi(x) = \{ \nabla \varphi(x) \}$ if $\varphi$ is differentiable at $x$, and $\partial \varphi(x) = \{ \nabla \varphi(x) \}$ holds when $\varphi$ is continuously differentiable around $x$ (\cite[ Exercise 8.8(b)]{rockafellar1998variational}). For a convex $\varphi$, the Fréchet and limiting subdifferential both consist with the classical subdifferential of a convex function at each $x \in \mathrm{dom}\varphi$ (\cite[ Proposition 8.12]{rockafellar1998variational}), that is,
    \[
    \widehat{\partial} \varphi(x) = \partial \varphi(x) = \{ y \in \mathbb{R}^n : \varphi(z) - \varphi(x) - \langle y, z - x \rangle \geq 0, \text{ for all } z \in \mathbb{R}^n \}.
    \]
    For proper functions $\varphi_1, \varphi_2 : \mathbb{R}^n \to (-\infty, +\infty]$, there holds $\widehat{\partial} \varphi_1(x) + \widehat{\partial} \varphi_2(x) \subseteq \widehat{\partial} (\varphi_1 + \varphi_2)(x)$ (\cite[ Corollary 10.9] {rockafellar1998variational}), where the equality holds if $\varphi_1$ or $\varphi_2$ is differentiable at $x$ (\cite[ Section 8.8 Exercise] {rockafellar1998variational}). For a function $\varphi : \mathbb{R}^n \to (-\infty, +\infty]$, the generalized \emph{Fermat's Rule} states that $0 \in \widehat{\partial} \varphi(x^*)$ holds if $x^* \in \mathrm{dom} \varphi$ is a local minimizer of $\varphi$ (\cite[ Theorem 10.1] {rockafellar1998variational}).\\
    \indent We finally review some properties of real-valued convex functions. Let $\varphi:\mathbb{R}^{n}\to\mathbb{R}$ be convex. Then for any bounded $\mathcal{S}\subseteq\mathbb{R}^{n}$, the $\varphi$ is Lipschitz continuous on $\mathcal{S}$ and $\cup_{x\in\mathcal{S}}\partial\varphi(x)$. The conjugate of $\varphi$, denote by $\varphi^{\star} : \mathbb{R}^n \to [-\infty, +\infty]$, is defined at $y\in\mathbb{R}^{n}$ as 
    $\varphi^{\star}(y):=\sup\{\langle x,y\rangle-\varphi(x):x\in\mathbb{R}^{n}\}$. The conjugate $\varphi^{\star}$ is a proper closed convex function \cite[ Theorem 12.2]{tyrrell1970convex}. For any $x,y\in\mathbb{R}^{n}$, the following three statements are equivalent (\cite[ Proposition 11.3]{rockafellar1998variational}: 
    $\langle x,y\rangle=\varphi(x)+\varphi^{\star}(y) \Leftrightarrow y\in\partial\varphi(x) \Leftrightarrow x\in\partial\varphi^{\star}(y)$. Therefore, it is clear that the \emph{Fenchel-Young Inequality}, i.e., $\varphi(x)+\varphi^{\star}(y)\geq\langle x,y\rangle$, holds for all $x,y\in\mathbb{R}^{n}$, and it becomes an equality when $y\in\partial\varphi(x)$. 
    
    \section{Solution existence of model \eqref{eq:problem1}}\label{sec3}
    \indent In this section, we investigate the existence of optimal solutions to model  \eqref{eq:problem1} under some suitable conditions. First, due to the fact that $q\geq0$ and $\|x\|_1^2/\|x\|_2\geq1$ for all $x\in\mathbb{R}^n$, we know that the infimum of model \eqref{eq:problem1} always exists and denote it by $\nu^{\star}$, i.e.,
    \[
     \nu^{\star}=\text{inf}\left\{\lambda\frac{\|x\|_1^2}{\|x\|_2}+q(Ax-b):x\in\mathcal{X}\right\}.
    \]
    One can easily check that
    \begin{equation}\label{eqt3.1}
    	\liminf_{\substack{x \to 0 \\ x \in \mathcal{X}}} \lambda\frac{\|x\|_1^2}{\|x\|_2}+q(Ax-b)=\lambda+q(-b).
    \end{equation}
    Invoking \eqref{eqt3.1} we immediately see that if the zero vector is not an optimal solution to model \eqref{eq:problem1} if and only if
    \begin{equation}\label{eqt3.2}
    	\nu^{\star}<\lambda+q(-b).
    \end{equation} 
    Hence, in order to preclude the trivial solution, we adopt the blanket assumption that \eqref{eqt3.2} holds for model \eqref{eq:problem1} in the rest of this paper. Next, we view the concepts of a minimizing sequence and spherical section property \cite{vavasis2009derivation,zhang2013theory}, which will be used in subsequent analysis in this section. Recall that $\left\{x^k:k\in\mathbb{N}\right\}$ is called a minimizing sequence of model \eqref{eq:problem1} if $x^k\in\mathcal{X}\setminus \{0\}$ for all $k\in\mathbb{N}$ and 
    \begin{equation}\label{eqt3.3}
    	\lim_{k\to\infty} \lambda \|x^k\|_1^2/\|x^k\|_2+q(Ax^k-b)=\nu^{\star}.
    \end{equation}  
    The spherical section property is defined as follows.
    \begin{definition}(spherical section property  \cite{vavasis2009derivation,zhang2013theory})
    	Let $m$, $n$ be two positive integers such that $m<n$. Let $V$ be an $n$-$m$-dimensional subspace of $\mathbb{R}^n$ and $s$ be a positive integer. We say that $V$ has the $s$-spherical section property if
    	\[
    	\underset{v\in V\setminus \{0\}}{\inf} \frac{\|v\|_1}{\|v\|_2} \geq \sqrt{\frac{m}{s}}.
    	\]
    \end{definition}
    According to \cite{zhang2013theory}, if $A\in\mathbb{R}^{m\times n} (m<n)$ is a random matrix with independently and identically distributed (i.i.d.) standard Gaussian entries, then its ($n$-$m$)-dimensional null space exhibits the $s$-spherical section property for $s=a_1(\log(n/m)+1)$ with probability at least $1-e^{-a_0(n-m)}$, where $a_0,a_1>0$ are constants independent of $m$ and $n$.\\
    \indent The following lemma concerns a minimizing sequence of model \eqref{eq:problem1}. 
    \begin{lemma}\label{lem:1}
    	Let \( \{x^k : k \in \mathbb{N}\} \) be a minimizing sequence of model \eqref{eq:problem1}. Then any accumulation point of \( \{x^k : k \in \mathbb{N}\} \) is an optimal solution of \eqref{eq:problem1}.
    \end{lemma}
    \begin{proof}
    	By passing to a subsequence if necessary, we may assume without loss of generality that \( \lim_{k \to \infty} x^k = x^{\star} \) for some \( x^{\star} \). Then it follows from the closedness of \( \mathcal{X} \) that $x^{\star}\in\mathcal{X}$. Invoking \eqref{eqt3.1}, \eqref{eqt3.2} and \eqref{eqt3.3}, we know that $x^{\star}\neq0$. Using this, the continuity of the objective function of model \eqref{eq:problem1} on $\mathcal{X}\setminus \{0\}$ and \eqref{eqt3.3}, we conclude that $\lambda\|x^{\star}\|_1^2/\|x^{\star}\|_2 + q(Ax^{\star}-b)=\nu^{\star}$. This complete the proof.
    \end{proof}
    \indent Now, we are ready to establish the main theorem of this section.
    	\begin{theorem}\label{the:2.1}
    		The set of optimal solutions to model \eqref{eq:problem1} is nonempty if one of the following statements holds:
    		\begin{enumerate}[label=\textup{(\roman*)}]
    		\item \( \mathcal{X} \) is bounded;\label{en:2.1}
    		\item \( q \) is coercive, \( \ker A \) has the \( s \)-spherical section property for some \( s > 0 \), and there exists \(\widetilde{x}\in\mathbb{R}^n\) such that \( 0 \neq \widetilde{x} \in \argmin_{x \in \mathcal{X}} q(Ax - b) \) and \( \|\widetilde{x}\|_0 <m/s \). \label{en:2.2}
    	    \end{enumerate}
        \end{theorem}
    \begin{proof}
    	Item \ref{en:2.1} is a direct consequence of Lemma \ref{lem:1} and the compactness of $\mathcal{X}$. We next focus on Item \ref{en:2.2}. Let $\{x^k : k \in \mathbb{N} \}$ be a minimizing sequence of problem \eqref{eq:problem1}. In view of Lemma \ref{lem:1}, we can prove this theorem by showing that $\{x^k : k \in \mathbb{N} \}$ is bounded. We next show this result by contradiction.\\
    	\indent Suppose $\{x^k : k \in \mathbb{N} \}$ is unbounded. Since $\frac{\|\cdot\|_1^2}{\|\cdot\|_1^2}$ is bounded, without loss of generality, we assume $ \lim_{k \to \infty} \frac{\lambda\|x^k\|_1^2}{\|x^k\|_2^2}= \nu_{1}^{\star}$. Then, $\lim_{k \to \infty}q(Ax^k - b) = \nu^{\star} - \nu_{1}^{\star}$. Let $x^k$ be decomposed as, for $k \in \mathbb{N}$, $x^k = x_{1}^k + x_{2}^k$ with $x_{1}^k \in \operatorname{ker} A$ and $x_{2}^k \in \operatorname{range} \, A^{T}$. The coercivity of $q$ indicates that $\{x_{2}^k : k \in \mathbb{N} \}$ is bounded. Then $\{x_{1}^k : k \in \mathbb{N} \}$ is unbounded, since $\{x^k:k \in \mathbb{N} \}$ is unbounded. 
    	Let $\tilde{x}^k := \frac{x^k}{\|x_{1}^k\|_2}$ for $k \in \mathbb{N}$. Then, there exists $\{\tilde{x}^{k_{j}} : j \in \mathbb{N} \}$ such that  $ \lim_{j \to \infty}\tilde{x}^{k_{j}} = \lim_{j \to \infty} \frac{x_{1}^{k_{j}}}{\|x_{1}^{k_{j}}\|_2} = d$, with $\|d\|_2 = 1$ and $d \in \operatorname{ker} A$. Then, 
    	\begin{equation}
    		\underset{j \to \infty}{\lim} \frac{\|x^{k_{j}}\|_1^2}{\|x^{k_{j}}\|_2^2} = \underset{j \to \infty}{\lim} \frac{\|\tilde{x}^{k_{j}}\|_1^2}{\|\tilde{x}^{k_{j}}\|_2^2} = \frac{\|d\|_1^2}{\|d\|_2^2}.
    		\label{equa:2.6}
    	\end{equation}
    	Since $\operatorname{ker} A$ has the $s$-spherical section property, it follows that 
    	\begin{equation}
    		\frac{\|d\|_1^2}{\|d\|_2^2} \geq \frac{m}{s} \overset{(a)}{>} \|\tilde{x}\|_0 \overset{(b)}{\geq} \frac{\|\tilde{x}\|_1^2}{\|\tilde{x}\|_2^2},
    		\label{equa:2.7}
    	\end{equation}
    	where (a) holds by our assumption and (b) follows from the Cauchy-Schwarz inequality. On the other hand, $\tilde{x} \in \argmin_{x \in \mathcal{X}}\,q(Ax - b)$. This together with \eqref{equa:2.6} and \eqref{equa:2.7} yields
    	\[
    	\frac{\lambda\|\tilde{x}\|_1^2}{\|\tilde{x}\|_2^2} + q(A\tilde{x} - b) < \underset{j \to \infty}{\lim} \frac{\lambda\|x^{k_{j}}\|_1^2}{\|x^{k_{j}}\|_2^2} + q(Ax^{k_{j}} - b).
    	\]
    	This contradicts to the fact that $\{x^k : k \in \mathbb{N} \}$ is the minimizing sequence of problem \eqref{eq:problem1}. We then complete the proof.
    \end{proof}
    \indent Finally, we establish the solution existence of the concrete models \eqref{eq:1.2}, \eqref{eq:1.3} and \eqref{eq:1.4} with the help of Theorem \ref{the:2.1}. The results are presented in the following corollary.
    \begin{corollary}
    	The following statements hold:
    	\begin{enumerate}[label=\textup{(\roman*)}]
    		\item If $\mathcal{X}$ is bounded,   then the set of optimal solutions to each of models \eqref{eq:1.2}, \eqref{eq:1.3} and \eqref{eq:1.4} is  nonempty;\label{cor:2.2.1}
    		\item If $\mathrm{ker} A$ has the $s$-spherical section property for some $s>0$ and there exists $\widetilde{x} \in\mathbb{R}^n$ such that $0\neq\widetilde{x} \in \argmin_{x\in\mathcal{X}}\,q(Ax-b)$ and $\|\widetilde{x}\|_0\leq m/s$, then the sets of optimal solutions to each of models \eqref{eq:1.2} and \eqref{eq:1.3} is nonempty.\label{cor:2.2.2}
    	\end{enumerate} 	
    \end{corollary}
    \begin{proof}
    	Item \ref{cor:2.2.1} follows from Theorem \ref{the:2.1} \ref{en:2.1}. Item \ref{cor:2.2.2} holds thanks to Theorem \ref{the:2.1} \ref{en:2.2} and the coerceivity of $q$ in problems \eqref{eq:1.2} and \eqref{eq:1.3}.
    \end{proof}	 
    \section{A proximal algorithm for solving \eqref{eq:problem1}} \label{sec4}
    \indent In this section, we study the first order optimality condition for problem \eqref{eq:problem1}, and accordingly propose an efficient proximal algorithm with rigorous subsequential convergence analysis. We shall consider a more general setting, where the underlying problem is a fractional program in the form of
    \begin{equation}\label{eq:problem2}
    	\underset{x\in\mathbb{R}^n}{\min} \; \left\{\frac{f^2(x)}{g(x)}+h_1(x)-h_2(x) : x\in C\cap\Omega\right\}.	
    \end{equation}
    Here, $f, g:\mathbb{R}^n\to[0,+\infty)$, and $h_1 , h_2:\mathbb{R}^n\to\mathbb{R}$ are proper closed, $C\subseteq\mathbb{R}^n$ is closed convex and $C\cap\Omega \neq \varnothing$ with $\Omega:=\{x\in\mathbb{R}^n:g(x)\neq0\}$. In addition, we adopt the following blanket assumption for problem \eqref{equation1.1}.
    \begin{assumption}\quad\label{assu:3.1}
    	\begin{enumerate}[label=\textup{(\roman*)},labelwidth=2em, leftmargin=*, align=left, itemsep=0pt]
    		\item $f$ is convex and the proximal operator associated with $f+\iota_C$ can be evaluated.\label{item:1}
    		\item $g$ is locally Lipschitz differentiable on $\mathbb{R}^n$.\label{item:2}
    		\item $h_1$ is locally Lipschitz differentiable on $\mathbb{R}^n$ and $h_2$ is convex.\label{item:3}		
    	\end{enumerate}
    \end{assumption}
    Clearly, model \eqref{eq:problem1} is a special case of \eqref{eq:problem2} with $f(x)=\sqrt{\lambda}\|x\|_1$, $g(x)=\|x\|_2^2$, $h_1(x)=q_1(Ax-b)$,  $h_2(x)=q_2(Ax-b)$, $C=\mathcal{X}$ and $\Omega=\{x\in\mathbb{R}^n:x\neq0\}$. For convenience, we introduce the extended objective $F:\mathbb{R}^n\to (-\infty,+\infty]$ for problem \eqref{eq:problem2}, which is defined at $x\in \mathbb{R}^n$ by 
    \begin{equation}
    	F(x):=
    	\begin{cases}
    		\frac{f^2(x)}{g(x)} + h_1(x) - h_2(x), & \text{if} \quad x \in \Omega \cap C,\\
    		+\infty, & \text{else}.
    	\end{cases}       
    \end{equation}\label{eq:1.5}
    Then problem \eqref{eq:problem2} is equivalent to minimizing $F$ over $\mathbb{R}^n$.
    \subsection{First-order optimality condition and the proposed algorithm}\label{subsection 4.1}
    \indent We first present a first order necessary optimality condition for problem \eqref{eq:problem2}.
    \begin{theorem}
    	If $x^{\star}$ is a local minimizer of problem \eqref{eq:problem2}, then it holds with $c_{\star}=f(x^{\star})/g(x^{\star})$ that 
    	\begin{equation}
    		0 \in \partial(2c_{\star}f+\iota_{C})(x^{\star})- c_{\star}^2\nabla g(x^{\star}) + \nabla h_1(x^{\star})-\partial h_2(x^{\star}).
    		\label{eq:problem9}
    	\end{equation}
    	\label{theor:3.2}
    \end{theorem}
    \begin{proof}
    	Due to the continuity of $g$ around $x^{\star} \in \Omega$, there exists $\delta>0$ such that $g(x)>0$ for all $x\in \mathcal{B}(x^{\star},\delta)$. Set $\varphi(x):=f^2(x)/g(x)-2c_{\star}f(x)+c_{\star}^2g(x)$. We next show that $\varphi$ is differentiable at $x^{\star}$ and $\nabla \varphi(x^{\star}) = 0$. We obtain from $c_{\star} = f(x^{\star})/g(x^{\star})$ that, for $x\in \mathcal{B}(x^{\star},\delta)$,
    	\begin{equation}
    		\varphi(x) - \varphi(x^{\star})= \frac{f^2(x)}{g(x)}-2c_{\star}f(x)+c_{\star}^2g(x)
    		=\frac{(f(x)-c_{\star}g(x))^2}{g(x)}.
    		\label{eq:pro8}
    	\end{equation} 
    	In view of Assumption \ref{assu:3.1} \ref{item:2}, $g$ is locally Lipschitz continuous. This together with Assumption \ref{assu:3.1} \ref{item:1} and $c_{\star}=f(x^{\star})/g(x^{\star})$ yields that $\|f(x)-c_{\star}g(x)\|_2 = \|f(x)-c_{\star}g(x)-(f(x^{\star})-c_{\star}g(x^{\star}))\|_2 \leq a\|x-x^{\star}\|_2$ for some $a>0$. Invoking this and \eqref{eq:pro8} we obtain that $\nabla \varphi(x^{\star})=0$. 
    	\indent Then, we deduce that 
    	 \[
    		\begin{aligned}
    		\widehat{\partial}F(x^{\star})
    		&\stackrel{(a)}{=} \widehat{\partial}(F-\varphi)(x^{\star})\\
    		& = \widehat{\partial}(2c_{\star}f-c_{\star}^2g +h_1 -h_2+\iota_C+\iota_{\Omega})(x^{\star})\\
    		&\stackrel{(b)}{=} \widehat{\partial}(2c_{\star}f+\iota_{C}+\iota_{\Omega}-h_2)(x^{\star})- c_{\star}^2\nabla g(x^{\star}) + \nabla h_1(x^{\star})\\
    		&\stackrel{(c)}{\subseteq}\widehat{\partial}(2c_{\star}f+\iota_{C}+\iota_{\Omega})(x^{\star})- \partial h_2(x^{\star}) -c_{\star}^2\nabla g(x^{\star}) + \nabla h_1(x^{\star})\\
    		&\stackrel{(d)}{=}\partial(2c_{\star}f+\iota_{C})(x^{\star})- \partial h_2(x^{\star}) -c_{\star}^2\nabla g(x^{\star}) + \nabla h_1(x^{\star}),
    	\end{aligned}
    	\]
    	where (a) is from $\nabla\varphi(x^{\star})=0$, (b) follows from the differentiability of $g$ and $h_1$, (c) comes from \cite[Theorem 3.1]{mordukhovich2006frechet} and (d) is form the fact that $\Omega$ is an open subset. This together with the generalized Fermat's rule implies \eqref{eq:problem9} immediately.
    \end{proof}
    Inspired by Theorem \ref{theor:3.2}, we define a critical point of $F$ as follows.
     \begin{definition}\label{de:3.1}
    	We say $x^{\star} \in \Omega \cap C$ is a critical point of F, if \eqref{eq:problem9} holds with $c_{\star}=f(x^{\star})/g(x^{\star})$.
    \end{definition}   
    Recall that for a proper lower semicontinuous convex function $\varphi:\mathbb{R}^n\to(-\infty,+\infty]$, the proximity operator of $\varphi$, denoted by $\text{prox}_{\varphi}$, defined at $x\in\mathbb{R}^n$ as
    \[
     \text{prox}_{\varphi}(x):=\text{arg}\,\text{min}\left\{\varphi(x)+\frac{1}{2}\|x-y\|_2^2:y\in\mathbb{R}\right\}.
    \]
    It is direct to see that $x^{\star}$ is a critical point of F, if and only if $x^{\star}$ satisfies 
    \begin{equation}
    	x^{\star} \in \mathrm{prox}_{2\alpha c_{\star}f+\iota_{C}} ( x^{\star} - \alpha(\nabla h_1(x^{\star})-c_{\star}^2\nabla g^{\star}-z^{\star}) ),\label{eq:problem4.1}
    \end{equation}
    for some $\alpha>0$, $z^{\star}\in \partial h_2(x^{\star})$ and $c_{\star}=f(x^{\star})/g(x^{\star})$. Motivated by this observation, we propose a proximal algorithm for solving \eqref{subsection 4.1}, which is formally stated in Algorithm \ref{al:1}. Notably, Step 2 of Algorithm \ref{al:1} incorporates a line search scheme to determine a suitable stepsize $\alpha$, where the auxiliary function $ H: \mathbb{R}^n \times \mathbb{R}^n \times \mathbb{R} \to (-\infty,+\infty] $ is defined as 
    \begin{equation}
    	H(x,z,c):=2cf(x)+\iota_{C}(x)-c^2g(x) +h_1(x)+h_2^{\star}(z)-\langle x,z\rangle.
    	\label{equation:4.2} 
    \end{equation}
    Before proceeding, we make some interpretations on this line search scheme. Using \eqref{equation:4.2}, Fenchel-Young equality and the fact that $z^k\in \partial h_2(x^k)$, it is easy to verify that
    \[
      H(\widehat{x}^k,z^k,f(\widehat{x}^k)/g(\widehat{x}^k))=\frac{f^2(\widehat{x}^k)}{g(\widehat{x}^k)}+\iota_C(\widehat{x}^k)+h_1(\widehat{x}^k)-h_2(x^k)-\langle\widehat{x}^k-x^k,z^k \rangle.    
    \]
    Hence, in practice we do not need to compute the term $h_2^{\star}(z^k)$ in the line search scheme. Moreover, in view of the Fenchel-Young inequality, for all $x\in\Omega\cup C$, $z\in \text{dom} h_2^{\star}$ and $c_x=f(x)/g(x)$, it holds that 
    \begin{equation}\label{equat:4.7}
      F(x)\leq H(x,z,c_x).
    \end{equation}
    Therefore, the condition \eqref{eq:problem4.4} in the line search scheme implies the sufficient descent of F, i.e.,
    \begin{equation}\label{equa4.6}
    	F(\widetilde{x}^k)+\frac{\sigma}{2}\|\widetilde{x}^k-x^k\|_2^2\leq F(x^k).
    \end{equation}
    Although the subsequential convergence can be still guaranteed if the terminating condition \eqref{eq:problem4.4} in the Algorithm \ref{al:1} is replaced by \eqref{equa4.6}, it is generally difficult to show the sequential convergence of this altered algorithm under the KL assumption on F. In contrast, we can establish the sequential convergence of Algorithm \ref{al:1} equipped with line-search condition \eqref{eq:problem4.4} under the KL assumption on $H$, as it will be shown in section \ref{sec6}.\\
    \indent Now we specialize Algorithm \ref{al:1} to model \eqref{eq:problem1} and the resulting algorithm is presented in Algorithm \ref{al:2}. We remark that each iteration of Algorithm \ref{al:2} can be computed very efficiently. 
    \begin{algorithm}[H]
    	\caption{A fractional optimization method for solving problem \eqref{eq:problem2}}
    	\label{al:1}
    	\begin{algorithmic}[H]
    		\State \textbf{Step 0.} Input: $x^0 \in \text{dom}(F)$, $0 < \underline{\alpha} < \overline{\alpha}$, $\sigma > 0$, $0 < r < 1$, and set $k \leftarrow 0$.
    		
    		\State \textbf{Step 1.} Compute $c_k=f(x^k)/g(x^k)$.
    		\State \hspace{3.7em} Choose $z^k \in \partial h_2(x^k)$.
    		\State \hspace{3.7em} Set $\alpha:=\widetilde{\alpha}_k\in [\underline{\alpha}, \overline{\alpha}]$.
    		
    		\State \textbf{Step 2.} Compute
    		\begin{equation}
    			\widehat{x}^{k} \in \mathrm{prox}_{2\alpha c_{k}f+\iota_{C}} ( x^{k} - \alpha(\nabla h_1(x^k)-c_{k}^2\nabla{g(x^k)}-z^{k}) ).\label{eq:problem4.3}
    		\end{equation}
    		\State \hspace{3.7em} If $\widehat{x}^{k} \in \Omega$ and satisfies
    		\begin{equation}
    			H(\widehat{x}^{k},z^k,f(\widehat{x}^{k})/g(\widehat{x}^{k})) +\frac{\sigma}{2}\|\widehat{x}^k-x^k\|_2^2 \leq F(x^k);
    			\label{eq:problem4.4}
    		\end{equation}
    		\State \hspace{3.7em} Then, set $x^{k+1}=\widehat{x}^{k}$ and go to \textbf{Step 3};
    		\State \hspace{3.7em} Else, set $\alpha:=\alpha r$ and repeat \textbf{Step 2}.
    		\State \textbf{Step 3.} Record $\alpha_k:=\alpha$, set $k \leftarrow k+1$, and go to \textbf{Step 1}.
    	\end{algorithmic}
    \end{algorithm}
    \begin{algorithm}[H]
    	\caption{The specialization of Algorithm \ref{al:1} to model \eqref{eq:problem1}}
    	\label{al:2}
    	\begin{algorithmic}[H]
    		\State \textbf{Step 0.} Input: $x^0 \neq 0_n$, $0 < \underline{\alpha} < \overline{\alpha}$, $\sigma > 0$, $0 < r < 1$, and set $k \leftarrow 0$.
    		
    		\State \textbf{Step 1.} Compute $c_k=\sqrt{\lambda}\|x^k\|_1/\|x^k\|_2^2$.
    		\State \hspace{3.7em} Choose $z^k \in A^T\partial q_2(Ax^k-b)$.
    		\State \hspace{3.7em} Set $\alpha:=\widetilde{\alpha}_k\in [\underline{\alpha}, \overline{\alpha}]$.
    		
    		\State \textbf{Step 2.} Compute
    		\[
    		\widehat{x}^{k} \in \mathrm{prox}_{2\alpha c_{k}\sqrt{\lambda}\|\cdot\|_1+\iota_{\mathcal{X}}} ( x^{k} - \alpha A^T\nabla q_1(Ax^k-b)+2\alpha c_k^2 x^k +\alpha z^{k} ).
    		\]
    		\State \hspace{3.7em} If $\widehat{x}^{k} \neq 0$ and satisfies
    		\[
    		\begin{aligned}
    			\lambda\frac{\|\widehat{x}^k\|_1^2}{\|\widehat{x}^k\|_2} &+q_1(A\widehat{x}^k-b)-\left\langle\widehat{x}^k-x^k,z^k\right\rangle+\frac{\sigma}{2}\|\widehat{x}^k-x^k\|_2^2\\ 
    			&\leq \lambda\frac{\|x^k\|_1^2}{\|x^k\|_2}+q_1(Ax^k-b);
    		\end{aligned}
    		\]
    		\State \hspace{3.7em} Then, set $x^{k+1}=\widehat{x}^{k}$ and go to \textbf{Step 3};
    		\State \hspace{3.7em} Else, set $\alpha:=\alpha r$ and repeat \textbf{Step 2}.
    		\State \textbf{Step 3.} Record $\alpha_k:=\alpha$, set $k \leftarrow k+1$, and go to \textbf{Step 1}.
    	\end{algorithmic}
    \end{algorithm}
    \noindent In particular, the proximal operator associated with the function $2\alpha c_k \sqrt{\lambda}\|\cdot\|_1+\iota_{\mathcal{X}}$  admits a simple closed form solution. Let $\mathcal{X}:=\{x\in \mathbb{R}^n:\underline{x}\leq x\leq\overline{x}\}$ for some $\underline{x}\in \mathbb{R}^n \cup \{-\infty\}^n$ and $\overline{x}\in \mathbb{R}^n \cup \{\infty\}^n$. Then, it is easy to verify that for $j=1,2,\dots,n$,
    \[
      \left(\text{prox}_{2\alpha c_k \sqrt{\lambda}\|\cdot\|_1+\iota_{\mathcal{X}}}(z)\right)_j=
      \max\left\{\min\left\{\text{sign}(z_j) \max\left\{0,|z_j|-2\alpha c_k \sqrt{\lambda}\right\},\bar{x}_j\right\},\underline{x}_j\right\}.
    \]
    \subsection{The well-definedness of Algorithm \ref{al:1}}
    \indent We establish the well-definedness of Algorithm \ref{al:1} in this subsection, that is, Step 2 of Algorithm \ref{al:1} can be terminated in finite steps. To this end, we first introduce the following basic assumption concerning its initial point.
    \begin{assumption}
    	The level set $\mathcal{X}_0:=\left\{x\in \mathrm{dom}(F):F(x) \leq F(x^0) \right\}$ is compact. \label{ass:4.1}
    \end{assumption}
    \indent We point out that the boundedness of the level set of $F$ is a very standard assumption in nonconvex optimization and will be satisfied if $F$ is coerceive. The closedness of the level set of $F$ automatically holds if the extended objective $F$ is closed. Nevertheless, the $F$ defined in \eqref{eq:1.5} may be not closed due to the non-closedness of $\Omega$. In view of \cite[Proposition 4.4]{zhang2022first} and the closedness of $h_1, h_2$ and $C$, we see that $F$ is closed on $\mathbb{R}^n$ if $f$ and $g$ do not attain zero simultaneously. Unfortunately, this requirement is not satisfied for problem \eqref{eq:problem1}. In this case, we can still guarantee the closedness of $\mathcal{X}_0$ via appropriately selecting $x^0$, as shown in the following proposition.
    \begin{proposition}
    	Suppose that S:=$\left\{x\in\mathbb{R}^n:f(x)=g(x)=0\right\}\neq\emptyset$. Then, $\mathcal{X}_0$ is closed if $x^0\in \mathrm{dom}F$ satisfies
    	\begin{equation}
    		F(x^0) < \inf \left\{\underset{z \to x}{\liminf}F(z):x\in S\right\}. \label{eq.4.2}
    	\end{equation} 
    	\label{prop.4.2}
    \end{proposition}
    \begin{proof}
    	We prove this proposition by contradiction. Suppose there exists $\{x^k : k \in \mathbb{N}\} \subseteq \mathcal{X}_0$ such that $\lim_{k \to \infty}x^k=x^{\star}$ with $x^{\star} \notin \mathcal{X}_0$. Due to the continuity of $F$ on $\mathrm{dom}F$ and the closedness of $\mathcal{X}$, we deduce that $x^{\star} \notin \Omega$, that is $g(x^{\star}) = 0$. Since $x^k \in \mathcal{X}_0$, we have for any $k \in \mathbb{N}$,
    	\begin{equation}
    		f^2(x^k)+(h_1(x^k)-h_2(x^k))g(x^k)\leq F(x^0)g(x^k).
    		\label{eq.4.3}
    	\end{equation}
    	Passing to the limit on both sides of the above inequality with $k\to \infty$, we derive from the continuity of $f, h_1, h_2$ and $g$ that 
    	\[
    	f^2(x^{\star})+(h_1(x^{\star})-h_2(x^{\star}))g(x^{\star})\leq F(x^0)g(x^{\star}).
    	\]
    	The above inequality yields that $f(x^{\star}) = 0$ from the fact that $g(x^{\star}) = 0$, which indicates that $x^{\star} \in S$. This means that there exists $\{x^k : k \in \mathbb{N}\} \subseteq \mathcal{X}_0$ such that $\lim_{k \to \infty}x^k=x^{\star} \in S$. Thus, we have $\inf \left\{\liminf_{z \to x}F(z):x\in S\right\} \leq \liminf_{k \to \infty}F(x^k)\leq F(x^0)$, where the last inequality comes from $F(x^k) \leq F(x^0)$ for $k\in \mathbb{N}$. This contradicts to \eqref{eq.4.2}. Then we immediately complete the proof.
    \end{proof}
    \indent We remark that for problem \eqref{eq:problem1}, the zero vector is the unique point at which both $f$ and $g$ vanish. In view of Proposition \ref{prop.4.2} and \eqref{eqt3.1}, we obtain that for model \eqref{eq:problem1} if $0 \neq x^0 \in \mathcal{X}$ satisfies
    \begin{equation}\label{eq4.12}
    	F(x^0) < \lambda+q(-b),
    \end{equation}
    then the level set $\mathcal{X}_0$ is closed. The boundedness of $\mathcal{X}_0$ can be guaranteed if $\mathcal{X}$ is bounded.\\
    \indent Now, with the help of Assumption \ref{ass:4.1}, we can establish the following technical lemma, which will be frequently used in our convergence analysis.
    \begin{lemma} 
    	Suppose that Assumption \ref{ass:4.1} holds. Then, we have $m_g:=\inf\left\{g(x):x\in \mathcal{X}_0 \right\} >0$. Moreover, there exists $\Delta>0$ such that the following statements hold with the set $\mathcal{X}_{\Delta}$ defined by
    	\[
    	\mathcal{X}_{\Delta}:=\left\{x\in \mathbb{R}^n:\mathrm{dist}(x,\mathcal{X}_0) \leq \Delta\right\},
    	\]
    	\begin{enumerate}[label=\textup{(\roman*)}]
    		\item For any $x \in \mathcal{X}_{\Delta}$, there holds $g(x) \geq m_g/2$;\label{lem:4.1}
    		\item $f$, $g$, $f/g$, $\nabla g$ and $\nabla h_1$ are globally Lipschitz continuous on $\mathcal{X}_{\Delta}$;\label{lem:4.2}
    		\item The quantities defined below are finite:\\
    		$
    		M_{\nabla g}:=\sup\left\{\|\nabla g(x)\|_2:x \in \mathcal{X}_{\Delta}\right\},
    		$\\
    		$
    		M_g:=\sup\left\{g(x):x \in \mathcal{X}_{\Delta}\right\},
    		M_{f/g}:=\sup\left\{f(x)/g(x):x \in \mathcal{X}_{\Delta}\right\}.
    		$
    		\label{lem:4.3} 
    	\end{enumerate}
    	\label{lemm:4.1}
    \end{lemma}
    \begin{proof}
    	Since $g$ is continuous and $g>0$ on the compact set $\mathcal{X}_{0}$, we obtain that $m_g>0$ and Item \ref{lem:4.1}. Then, Item \ref{lem:4.2} follow from the local Lipschitz continuity around each point in the compact $\mathcal{X}_{\Delta}\subseteq \Omega$ \cite[chapter 1, 7.5 Exercise (c)]{2009Nonsmooth}. Lastly, $M_{\nabla g}$, $M_g$ and $M_{f/g}$ are finite due to the continuity of $\nabla g$, $g$ and $f/g$ on the compact $\mathcal{X}_{\Delta}$, respectively. We then complete the proof.
    \end{proof}
    \indent Next, we establish in the following proposition the well-definedness of Algorithm \ref{al:1}, that also implies its sufficient descent property.
    \begin{proposition}\label{propos:4.2}
    	Suppose that Assumption \ref{ass:4.1} holds. Then Algorithm \ref{al:1} is well defined. Specifically, there exists $\underline{\alpha}_{\sigma}>0$ such that for all $k\in \mathbb{N}$, Step 2 of Algorithm \ref{al:1} terminates at some $\alpha_k \geq \min\{\underline{\alpha},\underline{\alpha}_{\sigma} \}\gamma$. Moreover, the sequence $\{x^k:k\in \mathbb{N}\}$ generated by Algorithm \ref{al:1} falls into the level set $\mathcal{X}_{0} \subseteq \Omega\cap C$ and satisfies
    	\begin{equation}
    		F(x^{k+1}) +\frac{\sigma}{2}\|x^{k+1}-x^k\|_2^2 \leq F(x^k).\label{eq.4.7}
    	\end{equation}
    	\label{pro:4.3}
    \end{proposition}
    \begin{proof}
    	We prove this proposition by induction. It is clear that $x^0\in \mathcal{X}_0$ and we assume that $x^k\in \mathcal{X}_0$ for some $k\geq0$.\\
    	\indent Let $\mathcal{X}_{\Delta}$ be given as in Lemma \ref{lemm:4.1}. We first show that there exists $\underline{\alpha}_{\Delta}>0$ such that $\widehat{x}^k$ falls into $\mathcal{X}_{\Delta}\cap C$ as  $\alpha \in (0,\underline{\alpha}_{\Delta}]$. Invoking \eqref{eq:problem4.3}, we deduce from the definition of proximity operators that $\widehat{x}^k \in \mathcal{X}$ and
    	\begin{equation}
    		2\alpha c_k f(\widehat{x}^k)+ \frac{1}{2}\|\widehat{x}^k-x^k\|_2^2 +\langle \widehat{x}^k-x^k, \alpha(\nabla h_1(x^{k}) - c_{k}^2\nabla g(x^k) -z^k)\rangle \leq 2\alpha c_kf(x^k). 
    		\label{eq:4.8}
    	\end{equation}
    	This together with the Cauchy-Schwarz inequality yields that $2\alpha c_k f(\widehat{x}^k)+ \frac{1}{2}\|\widehat{x}^k-x^k\|_2^2 - \alpha\|\widehat{x}^k-x^k\|_2 \|\nabla h_1(x^{k}) - c_{k}^2\nabla g(x^k)-z^k\|_2 \leq 2\alpha c_kf(x^k)$. Noting that this is a quadratic inequality of the term $\|\widehat{x}^k-x^k\|_2$, we deduce from the non-negativity of $2\alpha c_kf(\widehat{x}^k)$ that
    	\[
    	\|\widehat{x}^k-x^k\|_2\leq \alpha \|\nabla{h_1(x^{k})} - c_{k}^2\nabla g(x^k) -z^k\|_2+\sqrt{\alpha^2 \|\nabla h_1(x^{k}) - c_{k}^2\nabla g(x^k) -z^k\|_2^2+4\alpha c_kf(x^k)}.  
    	\]
    	Since $\sup \left\{\|\nabla h_1(x) - c^2\nabla g(x) -z\|_2:x\in \mathcal{X}_0, z\in \partial h_2(x), c=f(x)/g(x)\right\}<+\infty$ and  
    	$\sup \\ \allowbreak  \left\{cf(x) : x \in \mathcal{X}_0, c=f(x)/g(x) \right \} <+ \infty $ due to the compactness of $\mathcal{X}_0$, the continuity of $\nabla{h_1}, f/g, \nabla g, f^2/g$ on $\mathcal{X}_0$ and the convexity of $h_2$, the above inequality along with $x^k \in \mathcal{X}_0$ implies that the term $\|\widehat{x}^k-x^k\|_2$ can be narrow down by the parameter $\alpha$. Therefore, there exists some $\underline{\alpha}_{\Delta}>0$ such that
    	\begin{equation}
    		\|\widehat{x}^k-x^k\|_2 \leq \Delta/2, \text{ for any } \alpha \in (0,\underline{\alpha}_{\Delta}].
    		\label{eq:4.9}
    	\end{equation}
    	This together with $x^k\in \mathcal{X}_0$ yields that $\widehat{x}^k \in \mathcal{X}_{\Delta}\cap\mathcal{X}$ when $\alpha \in (0,\underline{\alpha}_{\Delta}]$.\\
    	\indent Next, we shall show that there exists $\underline{\alpha}_{\sigma}\in(0,\underline{\alpha}_{\Delta}]$ such that $\widehat{x}^k\in \Omega$ and satisfies \eqref{eq:problem4.4} stated in Step 2 of Algorithm \ref{al:1} whenever $\alpha\in(0,\underline{\alpha}_{\sigma}]$. It is obvious that $\widehat{x}^k\in \Omega$ due to $\widehat{x}^k \in \mathcal{X}_{\Delta}$ and Lemma \ref{lemm:4.1} \ref{lem:4.1}. Moreover, $\widehat{x}^k\in\mathcal{X}_{\Delta}$ together with Lemma \ref{lemm:4.1} \ref{lem:4.2} yields that
    	\begin{equation}
    		h_1(\widehat{x}^k ) \leq h_1(x^k) + \langle\nabla h_1(x^{k}),\widehat{x}^k-x^k\rangle +\frac{L_{\nabla h_1}}{2}\|\widehat{x}^k-x^k\|_2^2,
    		\label{eq:4.10}
    	\end{equation}
    	\begin{equation}
    		-c_k^2g(\widehat{x}^k)\leq -c_k^2g(x^k)+\langle-c_k^2\nabla g(x^k), \widehat{x}^k-x^k\rangle+\frac{M_{f/g}^2L_{\nabla g}}{2}\|\widehat{x}^k-x^k\|_2^2,
    		\label{equ:4.12}
    	\end{equation}
    	where $L_{\nabla h_1}>0$ and $L_{\nabla g}>0$ denote the Lipschitz modulus of $\nabla h_1$ and $\nabla g$ on $\mathcal{X}_{\Delta}$, respectively. Multiplying $1/\alpha$ on both sides of \eqref{eq:4.8} and summing \eqref{eq:4.10} and \eqref{equ:4.12}, we have
    	\begin{align} \notag
    			&2c_k f(\widehat{x}^k)-c_k^2g(\widehat{x}^k)+\langle \widehat{x}^k-x^k,-z^k\rangle +h_1(\widehat{x}^k)+ (\frac{1}{2\alpha}-\frac{L_{\nabla h_1}}{2}-\frac{M_{f/g}^2L_{\nabla g}}{2}) \|\widehat{x}^k-x^k\|_2^2\\  &\leq 
    			2 c_kf(x^k) -c_k^2g(x^k) +h_1(\widehat{x}^k).
    		\label{eq:4.11}
    	\end{align}
    	Then, invoking $\langle x^k,z^k \rangle = h_2(x^k)+h_2^{\star}(z^k)$ due to $z^k \in \partial h_2(x^k)$, we deduce from \eqref{eq:4.11} and $c_k=f(x^k)/g(x^k)$ that 
    	\begin{align} \notag
    			&2c_k f(\widehat{x}^k)-c_k^2g(\widehat{x}^k)+h_1(\widehat{x}^k)-\langle \widehat{x}^k,z^k \rangle+h_2^{\star}(z^k)+ (\frac{1}{2\alpha}-\frac{L_{\nabla h_1}}{2}-\frac{M_{f/g}^2L_{\nabla g}}{2})) \|\widehat{x}^k-x^k\|_2^2 \\ &\leq 
    			2c_kf(x^k)-c_k^2g(x^k)+h_1(x^k)-h_2(x^k)=F(x^k).
    		\label{eq:4.14}
    	\end{align}
    	By simple calculation, we have for $\widehat{c}:=f(\widehat{x}^k)/g(\widehat{x}^k)$
    	\begin{align}\notag
    			0 &\leq 2\widehat{c}f(\widehat{x}^k)-\widehat{c}^2g(\widehat{x}^k) - (2c_kf(\widehat{x}^k)-c_k^2g(\widehat{x}^k))
    			=\frac{f^2(\widehat{x}^k)}{g(\widehat{x}^k)} - \left(2c_kf(\widehat{x}^k)-c_k^2g(\widehat{x}^k)\right)\\
    			&=g(\widehat{x}^k)\left(c_k-\frac{f(\widehat{x}^k)}{g(\widehat{x}^k)}\right)^2
    			\leq M_gL^2_{f/g}\|\widehat{x}^k-x^k\|_2^2,
    		\label{eq:4.15}	    	
    	\end{align}
    	where the first inequality holds since $\widehat{c} = \arg\max\left\{2cf(\widehat{x}^k)-c^2g(\widehat{x}^k):c\in \mathbb{R}\right\}$ and the last inequality holds from the fact $\widehat{x}^k \in \mathcal{X}_{\Delta}$, Lemma \ref{lemm:4.1} \ref{lem:4.3} and the Lipschitz continuity of $f/g$ on $\mathcal{X}_{\Delta}$(Lemma \ref{lemm:4.1} \ref{lem:4.2}) with modulus $L_{f/g}$. Then \eqref{eq:4.14} and \eqref{eq:4.15} imply that 
    	\begin{equation}
    		H(\widehat{x}^k,z^k,f(\widehat{x}^k)/g(\widehat{x}^k))+\left(\frac{1}{2\alpha}-\frac{L_{\nabla h_1}}{2}-\frac{M_{f/g}^2L_{\nabla g}}{2}-M_gL^2_{f/g}\right)\|\widehat{x}^k-x^k\|_2^2\leq F(x^k).
    		\label{eq:4.16}
    	\end{equation}
    	This together with \eqref{equat:4.7} leads
    	to 
    	\begin{equation}
    		F(\widehat{x}^{k}) +\left(\frac{1}{2\alpha}-\frac{L_{\nabla h_1}}{2} -\frac{M_{f/g}^2L_{\nabla g}}{2} - M_gL^2_{f/g}\right) \|\widehat{x}^k-x^k\|_2^2 \leq F(x^k).
    		\label{eq:4.17}
    	\end{equation} 
    	Therefore, by setting $\underline{\alpha}_{\sigma}:=\min\left\{\underline{\alpha}_{\Delta}, \left(L_{\nabla h_1}+M_{f/g}^2L_{\nabla g}+2M_gL^2_{f/g}\right)^{-1}\right\}$, we deduce $\widehat{x}^k \in \mathcal{X}_{0}$ and \eqref{eq.4.7} from \eqref{eq:4.17} and Step 2 of Algorithm \ref{al:1} terminates at some $\alpha_k \geq \min \left\{\underline{\alpha},\underline{\alpha}_{\sigma}\right\}\gamma$ in the $k$-th iteration from \eqref{eq:4.16}. We complete the proof immediately.  
    \end{proof}
    \subsection{Subsequential convergence of Algorithm \ref{al:1}}
    We investigate in this subsection the subsequential convergence of Algorithm \ref{al:1} under Assumption \ref{ass:4.1}. We will show in Theorem \ref{the:4.1} that any accumulation point of sequence generated by Algorithm \ref{al:1} is a critical point of $F$. To this end, we first show the convergence of the objective values in the next corollary, which is a direct consequence of Proposition \ref{pro:4.3}.
    \begin{corollary}
    	Suppose that Assumption \ref{ass:4.1} holds. Let $\left\{(x^k,z^k,c_k):k\in\mathbb{N}\right\}$ be generated by Algorithm \ref{al:1}. Then the following statements holds:
    	\begin{enumerate}[label=\textup{(\roman*)}]
    		\item $\left\{(x^k,z^k,c_k):k\in\mathbb{N}\right\}$ is bounded;
    		\label{cor:4.1}
    		\item $\underset{k \to \infty}{\lim}F(x^k)$ exists;
    		\label{cor:4.2}
    		\item $\underset{k \to \infty}{\lim}\|x^{k+1}-x^{k}\|_2=0$.
    		\label{cor:4.3} 
    	\end{enumerate}
    	\label{corollary:4.1}
    \end{corollary}
    \begin{proof}
    	Firstly, Proposition \ref{pro:4.3} implies that $\left\{x^k:k\in\mathbb{N} \right\}\subseteq \mathcal{X}_{0}$. This together with Assumption \ref{ass:4.1} leads to the boundedness of $\left\{x^k:k\in\mathbb{N}\right\}$. Since $h_2$ is real-valued convex, the boundedness of $\left\{x^k:k\in\mathbb{N}\right\}$ as well as $z^k \in \partial h_2(x^k)$ yield that $\{z^k:k\in\mathbb{N}\}$ is bounded. Lemma \ref{lemm:4.1} \ref{lem:4.3}  and $\{x^k:k\in\mathbb{N}\}\subseteq\mathcal{X}_{0}\subseteq\mathcal{X}_{\Delta}$ imply that $\{c_k:k\in\mathbb{N}\}$ is bounded due to $c_k=f(x^k)/g(x^k)$. Then, Item \ref{cor:4.1} of this corollary is followed. Secondly, since $F$ is continuous on $\mathcal{X}_{0}$, $F$ is bounded below on the compact set $\mathcal{X}_{0}$. This together with \eqref{eq.4.7} indicates Item \ref{cor:4.2}. By passing to the limit on the both sides of \eqref{eq.4.7} with $k\to \infty$, Item \ref{cor:4.3} follows immediately.
    \end{proof}
    Now, we are ready to show the subsequential convergence of the proposed Algorithm \ref{al:1}.
    \begin{theorem}\label{the:4.1}
    	Suppose that Assumption \ref{ass:4.1} holds. Then any accumulation point of the sequence $\left\{x^k:k\in\mathbb{N}\right\}$ generated by Algorithm \ref{al:1} falls into $\mathcal{X}_{0}$, and is a critical point of $F$. 
    \end{theorem}
    \begin{proof}
    	Let $x^{\star}$ be an accumulation point of $\left\{x^k:k\in\mathbb{N}\right\}$. Since $\left\{x^k:k\in\mathbb{N}\right\} \subseteq \mathcal{X}_{0}$ due to Proposition \ref{pro:4.3}, we have $x^{\star}\in\mathcal{X}_{0}$ from the compactness of $\mathcal{X}_{0}$. We next show $x^{\star}$ is a critical point of $F$.\\
    	\indent In view of Proposition \ref{pro:4.3} and Algorithm \ref{al:1}, we have 
    	\begin{equation}
    		x^{k+1} \in \mathrm{prox}_{2\alpha_k c_{k}f+\iota_{\mathcal{X}}} ( x^{k} - \alpha_k(\nabla h_1(x^k)-c_{k}^2\nabla{g(x^k)}-z^{k}) )
    		\label{eq:4.18}
    	\end{equation}
    	for some $\alpha_k \geq \min\left\{\underline{\alpha},\underline{\alpha}_{\sigma}\right\}\gamma$. Since $\left\{(x^k,z^k,c_k):k\in\mathbb{N}\right\}$ is bounded (Corollary \ref{corollary:4.1}) and $\left\{\alpha_k:k\in\mathbb{N}\right\}$ is bounded, there exists $\{k_j:j\in \mathbb{N}\} \subseteq \mathbb{N}$, $z^{\star}\in\mathbb{R}^n$, $c_{\star} \geq 0$ and $\alpha_{\star}>0$, such that $\lim_{k \to \infty}x^{k_j}=x^{\star}$, $\lim_{k \to \infty}z^{k_j}=z^{\star}$, $\lim_{k \to \infty}c_{k_j}=c_{\star}$ and $\lim_{k \to \infty}\alpha_{k_j}=\alpha_{\star}$. Then, Corollary \ref{corollary:4.1} \ref{cor:4.3} implies that  $\lim_{k \to \infty}x^{k_j+1}=x^{\star}$. In view of the definition of the proximity operator, we deduce from \eqref{eq:4.18} that for any $x\in C$, it holds that
    	\begin{equation*}
    		\begin{aligned}
    			2c_{k_j} f(x^{k_j+1})+ \frac{1}{2\alpha_{k_j}}\|x^{k_j+1}-x^{k_j}+\alpha_{k_j}(\nabla h_1(x^{k_j}) - c_{k_j}^2 \nabla g(x^{k_j}) -z^{k_j})\|_2^2 \\
    			\leq 2 c_{k_j} f(x)+ \frac{1}{2\alpha_{k_j}}\|x-x^{k_j}+\alpha_{k_j}(\nabla h_1(x^{k_j}) - c_{k_j}^2 \nabla g(x^{k_j}) -z^{k_j})\|_2^2. 
    		\end{aligned}
    	\end{equation*}
    	By passing to the limit with $j\to\infty$, we obtain from the continuity of $f$, $\nabla h_1$ and $\nabla g$ that
    	\[
    	2c_{\star}f(x^{\star}) \leq 2c_{\star}f(x)+ \frac{1}{2\alpha_{\star}}\|x-x^{\star}\|_2^2 +\langle x-x^{\star}, \nabla h_1(x^{\star}) - c_{\star}^2\nabla g(x^{\star}) -z^{\star}\rangle 
    	\]
    	holds for any $x\in C$. This means that $x^{\star}$ is a minimizer of the problem
    	\[
    	\min\left\{2c_{\star}f(x)+ \frac{1}{2\alpha_{\star}}\|x-x^{\star}\|_2^2 +\langle x-x^{\star}, \nabla h_1(x^{\star}) - c_{\star}^2\nabla g(x^{\star}) -z^{\star}\rangle:x\in \mathcal{X} \right\}.   		
    	\]
    	This together with the generalized Fermat's Rule indicates that 
    	\begin{equation}
    		0\in \partial(2 c_{\star}f+\iota_{C})(x^{\star}) +\nabla h_1(x^{\star}) - c_{\star}^2\nabla g(x^{\star}) -z^{\star} .
    		\label{eq:4.19}
    	\end{equation}
    	Due to the continuity of $f/g$ on $\mathcal{X}_0$ and the closedness of $\partial h_2$, we have $c_{\star}=f(x^{\star})/g(x^{\star})$ and $z^{\star} \in \partial h_2(x^{\star})$. Then,  \eqref{eq:4.19} implies that $x^{\star}$ is a critical point of $F$. We then complete the proof.
    \end{proof}
    \section{Sequential convergence of Algorithm \ref{al:1}} \label{sec6}
    \indent In this section, we investigate the convergence of the entire sequence $\{x^k:k\in\mathbb{N}\}$ generated by the Algorithm \ref{al:1} for solving problem \eqref{eq:problem2} under the KL assumption. We first review the notion of KL property in the next subsection and then establish the convergence of the whole sequence $\{x^k:k\in\mathbb{N}\}$ generated by Algorithm \ref{al:1} in Section \ref{Se:5.2}.
    \subsection{KL property}
    We review in this subsection the KL property. This property and the associated notion of KL exponent have been used extensively in the convergence analysis of various first order methods \cite{attouch2009convergence,attouch2010proximal,attouch2013convergence,bolte2014proximal}.
    \begin{definition}[Kurdyka-Łojasiewicz (KL) property]
    	We say that a proper function \( \varphi : \mathbb{R}^n \to (-\infty, +\infty] \) satisfies the Kurdyka-Łojasiewicz (KL) property at an \(\widehat{x} \in \operatorname{dom} \partial \varphi\) if there are \(a \in (0, +\infty]\), a neighborhood \(U\) of \(\widehat{x}\) and a continuous concave function \(\phi : [0, a) \to [0, +\infty)\) with \(\phi(0) = 0\) such that:
    	\begin{enumerate}[label=\textup{(\roman*)}]
    		\item \(\phi\) is continuously differentiable on \((0, a)\) with \(\phi' > 0\) on \((0, a)\);
    		\item for every \(x \in U\) with \(\varphi(\widehat{x}) < \varphi(x) < \varphi(\widehat{x}) + a\), it holds that
    		\begin{equation}
    			\phi'(\varphi(x) - \varphi(\widehat{x})) \operatorname{dist}(0, \partial \varphi(x)) \geq 1.
    			\label{equ:5.1}
    		\end{equation}
    	\end{enumerate}
    \end{definition}
    A proper function $\varphi:\mathbb{R}^n\to(-\infty,+\infty]$ is called a KL function, if it satisfies the KL property at each point of $\textnormal{dom}\partial \varphi$. KL functions encompass extensive functions arose in various applications. For instance a subanalytic function \cite[Definition 6.6.1]{facchinei2003finite} $\varphi$, which is continuous around $x\in\textnormal{dom}\partial \varphi$, satisfies the KL property at $x$  \cite[Theorem 3.1]{bolte2007lojasiewicz}. Moreover, a proper semi-algebraic function $ \varphi$, which is lower semicontinuous around $x\in\textnormal{dom}\partial \varphi$, satisfies the KL property at $x$.\\
    \indent We next review a framework for proving sequential convergence using the KL property.
    \begin{proposition}{\cite[Proposition 2.7]{li2022proximal}}
    	Let \(\Psi : \mathbb{R}^n \times \mathbb{R}^m \to (-\infty, +\infty]\) be proper lower semicontinuous. Consider a bounded sequence \(\{(u^k, v^k) \in \mathbb{R}^n \times \mathbb{R}^m : k \in \mathbb{N}\}\) satisfying the following three conditions:
    	\begin{enumerate}[label=\textup{(\roman*)}]
    		\item (Sufficient descent condition). There exists \(a > 0\), such that
    		\[
    		\Psi(u^{k+1}, v^{k+1}) + a\|u^{k+1} - u^k\|_2^2 \leq \Psi(u^k, v^k), \quad \text{for all } k \in \mathbb{N};
    		\]
    		
    		\item (Relative error condition). There exists \(b > 0\), such that
    		\[
    		\mathrm{dist}(0, \partial \Psi(u^{k+1}, v^{k+1})) \leq b\|u^{k+1} - u^k\|_2, \quad \text{for all } k \in \mathbb{N};
    		\]
    		
    		\item (Continuity condition). The limit \(\Psi_{\infty} := \lim_{k \to \infty} \Psi(u^k, v^k)\) exists and \(\Psi \equiv \Psi_{\infty}\) holds on \(\Upsilon\), where \(\Upsilon\) denotes the set of accumulation points of \(\{(u^k, v^k) : k \in \mathbb{N}\}\). 
    	\end{enumerate}
    	If \(\Psi\) satisfies the KL property at each point of \(\Upsilon\), then we have \(\sum_{k=0}^{\infty} \|u^{k+1} - u^k\|_2 < +\infty\), \(\lim_{k \to \infty} u^k = u^*\) and \(0 \in \partial \Psi(u^*, v^*)\) for some \((u^*, v^*) \in \Upsilon\).
    	\label{prop:5.1}
    \end{proposition}
    \subsection{Sequential convergence of Algorithm \ref{al:1}}\label{Se:5.2}
    This subsection is devoted to the full sequential convergence of Algorithm \ref{al:1}. We begin with a key proposition regarding the sufficient descent of $H$ defined by  \eqref{equation:4.2} and the relative error condition. 
    \begin{proposition}
    	Suppose that Assumption \ref{ass:4.1} holds and let $\{(x^k,z^k,c_k):k\in \mathbb{N}\}$ be generated by Algorithm \ref{al:1}. Then, the following statements hold:
    	\begin{enumerate}[label=\textup{(\roman*)}]
    		\item There exists $a>0$, such that for all $k \in \mathbb{N}$,\\
    		$H(x^{k+1},z^k,c_{k+1}) + a\|x^{k+1}-x^k\|_2^2 \leq H(x^{k},z^{k-1},c_{k})$;
    		\label{pro:5.1}
    		\item There exists $b>0$, such that for all $k \in \mathbb{N}$,\\
    		$\mathrm{dist}\text{ }(0,\partial H(x^{k+1},z^k,c_{k+1})) \leq b\|x^{k+1}-x^k\|_2 $.
    		\label{pro:5.2} 
    	\end{enumerate}
    	\label{proposition:5.1}
    \end{proposition}
    \begin{proof}
    	We first prove Item \ref{pro:5.1}. Thanks to Proposition \ref{propos:4.2}, we know that \eqref{eq:problem4.4} holds with $\widehat{x}^k=x^{k+1}$ for any $k
    	\in \mathbb{N}$. This together with \eqref{equat:4.7} implies Item \ref{pro:5.1} for $a=\sigma/2$.\\
    	\indent We next prove Item \ref{pro:5.2}. We first show the following inclusion holds
    	\begin{equation}
    		\partial H(x^{k+1},z^k,c_{k+1}) \supseteq 
    		\begin{bmatrix}
    			&2c_{k+1}\partial(f+\iota_{C})(x^{k+1}) +\nabla h_1(x^{k+1}) - c_{k+1}^2\nabla g(x^{k+1}) -z^{k}\\
    			&\partial h_2^{\star}(z^k)-x^{k+1}\\
    			& 2f(x^{k+1})-2c_{k+1}g(x^{k+1})
    		\end{bmatrix}.
    		\label{eq:5.1}
    	\end{equation}
    	Set $\widetilde{H}(x,c,z)=H(x,z,c)$ for $(x,z,c)\in C\times\mathrm{dom}\partial h_2^{\star}\times (0,+\infty)$. Then, we derive from \cite{attouch2010proximal} that
    	\begin{equation}\label{equat:5.3}
    		\partial\widetilde{H}(x,c,z)=\partial(2cf(x)+\iota_{C}(x)-c^2g(x)+h_1(x))(x,c)\times \partial h_2^{\star}(z)-\begin{bmatrix}z \\ 0 \\ x \end{bmatrix}.
    	\end{equation}
    	The continuous differentiability of $g$ and $h_1$ leads to
    	\begin{equation}\label{equat:5.4}
    		\partial(2cf(x)+\iota_{C}(x)-c^2g(x)+h_1(x))(x,c)=\partial(2cf(x)+\iota_{C}(x))(x,c)+\begin{bmatrix}\nabla h_1(x)-c^2\nabla g(x) \\-2cg(x) \end{bmatrix}.
    	\end{equation}
    	According to \cite[Theorem 10.6]{rockafellar1998variational},  $\widehat{\partial}(2cf(x)+\iota_{C}(x))(x,c)=S_1\times \{2f(x)\}$ with $S_1\subseteq \mathbb{R}^n$. Due to the continuity of $f$ and the definition of limiting subdifferential, we have that
    	\begin{equation}\label{equat:5.5}
    		\partial(2cf(x)+\iota_{C}(x))(x,c)=S_2\times \{2f(x)\}
    	\end{equation}
    	with $S_2\subseteq\mathbb{R}^n$. Now, in view of \cite[Corollary 10.11]{rockafellar1998variational}, we deduce
    	\begin{equation}\label{equat:5.6}
    		S_2\supseteq \partial_x(2cf(x)+\iota_{C}(x)).
    	\end{equation}
    	Then, \eqref{equat:5.3}, \eqref{equat:5.4}, \eqref{equat:5.5} and \eqref{equat:5.6} yield \eqref{eq:5.1} from the fact that $\alpha \iota_{C}=\iota_{C}$ for any $\alpha>0$.\\
    	\indent It is known from the proximal Step \eqref{eq:problem4.3} of Algorithm \ref{al:1} and the generalized Fermat's Rule that
    	\begin{equation}
    		\omega^{k+1}_1:=\frac{x^{k}-x^{k+1}}{2\alpha_kc_k}-\frac{\nabla h_1(x^{k}) - c_{k}^2\nabla g(x^k) -z^{k}}{2c_k} \in \partial(f+\iota_{C})(x^{k+1}),
    		\label{eq:5.2}
    	\end{equation}
    	due to $\alpha\iota_{C}=\iota_{C}$ for any $\alpha>0$. Then, with the help of \eqref{eq:5.1} as well as $z^k \in \partial h_2(x^k)$, we obtain from $c_{k+1}=f(x^{k+1})/g(x^{k+1})$ that
    	\[
    	\omega^{k+1}:=(\omega^{k+1}_x,x^{k}-x^{k+1},0) \in \partial H(x^{k+1},z^k,c_{k+1}),
    	\]
    	where $\omega^{k+1}_x$ is given as 
    	\begin{equation}
    		\begin{aligned}
    			\omega^{k+1}_x&:=2c_{k+1}\omega^{k+1}_1 -c^2_{k+1}\nabla g(x^{k+1})+\nabla h_1(x^{k+1})-z^k.\\
    		\end{aligned}
    		\label{eq:5.3}
    	\end{equation}
    	\indent We next show $\|\omega^{k+1}\|_2\leq b\|x^{k+1}-x^k\|_2$ for some $b>0$. It suffices to prove that $\omega^{k+1}_x$ can be bounded by the term $\|x^{k+1}-x^k\|_2$. In view of \eqref{eq:5.2} and \eqref{eq:5.3}, we rewrite $\omega^{k+1}_x$ as
    	\[
    	\begin{aligned}
    		\omega^{k+1}_x
    		&=\frac{c_{k+1}}{\alpha_kc_k}(x^{k}-x^{k+1})+(\nabla h_1(x^{k+1})-\frac{c_{k+1}}{c_k}\nabla h_1(x^{k}))\\
    		&+c_{k+1}(c_k\nabla g(x^k)-c_{k+1}\nabla g(x^{k+1}))+(\frac{c_{k+1}}{c_k}-1)z^k. 
    	\end{aligned} 
    	\]
    	Due to the Lipschitz continuity of $f/g$ on the compact set $\mathcal{X}_{0}$, we set $M_{f/g}$, $m_{f/g}$ and $L_{f/g}$ be the supremum, minimum and the Lipschitz modulus of $f/g$ on $\mathcal{X}_{0}$, respectively. We also have $M_{\partial h_2}:=\sup\{\|z\|_2:z\in\partial h_2(x),x\in\mathcal{X}_{0}\}$ is finite from the convexity of the real valued function $h_2$ and the compactness of $\mathcal{X}_{0}$. Then, we deduce
    	\[
    	\begin{aligned}
    		\left\|\frac{c_{k+1}}{\alpha_kc_k}(x^k-x^{k+1})\right\|_2 &\leq \frac{M_{f/g}}{\min\left\{\underline{\alpha},\underline{\alpha}_{\sigma}\right\}\gamma \,m_{f/g}}\|x^{k+1}-x^k\|_2,\\
    		\left\|(\frac{c_{k+1}}{c_{k}}-1)z^k\right\|_2
    		&\leq \frac{L_{f/g}M_{\partial h_2}}{m_{f/g}}\|x^{k+1}-x^k\|_2.
    	\end{aligned}     
    	\]
    	Moreover, we have
    	\begin{align}\notag
    			\|\nabla h_1(x^{k+1})-\frac{c_{k+1}}{c_k}\nabla h_1(x^{k})\|_2
    			&=\frac{1}{c_k}\|c_k\nabla h_1(x^{k+1})-c_k\nabla h_1(x^{k})+c_k\nabla h_1(x^{k})-c_{k+1}\nabla h_1(x^{k})\|_2\\
    			&\leq (L_{\nabla h_1}+M_{\nabla h_1}L_{f/g}/m_{f/g}) \|x^{k+1}-x^k\|_2,        	
    	\end{align}
    	where $L_{\nabla h_1}$ is the Lipschitz modulus of $\nabla h_1$ on the compact $\mathcal{X}_{0}$, and $M_{\nabla h_1}:=\sup\{\|\nabla h_1(x)\|_2:x\in\mathcal{X}_{0}\}$. In addition,
    	\[
    	\|c_{k+1}(c_k\nabla g(x^k)-c_{k+1}\nabla g(x^{k+1}))\|_2\leq M_{f/g}(M_{f/g}L_{\nabla g}+M_{\nabla g} L_{f/g})\|x^{k+1}-x^k\|_2
    	\]  
    	with $L_{\nabla g}$ being the Lipschitz modulus of $\nabla g$ on $\mathcal{X}_0$. These imply that $\|\omega^{k+1}_x\|_2 \leq b_x\|x^{k+1}-x^k\|_2$ for some $b_x>0$. Then, we complete the proof immediately.
    \end{proof}
    \indent Now, we are ready to present the main result of this section in the following theorem.
    \begin{theorem}\label{theo5.1}
    	Suppose that Assumption \ref{ass:4.1} holds. If $H$ defined by \eqref{equation:4.2} is a proper closed KL function, then the sequence $\{x^k:k\in \mathbb{N}\}$ generated by Algorithm \ref{al:1} converges to a critical point of $F$.
    \end{theorem}
    \begin{proof}
    	Let the sequence $\{(x^k,z^k,c_k):k\in \mathbb{N}\}$ be  generated by Algorithm \ref{al:1}. Then, in view of Proposition \eqref{prop:5.1}, Corollary \ref{corollary:4.1}, Theorem \ref{eq:problem4.1} and Proposition \ref{proposition:5.1}, it remains to prove that $H_{\infty}:=\lim_{k \to \infty}H(x^{k+1},z^k,c_{k+1})$ exists and $H\equiv H_{\infty}$ on the set of all accumulation points of $\{(x^{k+1},z^k,c_{k+1}):k\in\mathbb{N}\}$. Invoking the Step 2 \eqref{eq:problem4.4} of Algorithm \ref{al:1} and \eqref{equat:4.7}, we have $F(x^{k+1})\leq H(x^{k+1},z^k,c_{k+1})\leq F(x^k)$. This together with Corollary \ref{corollary:4.1} \ref{cor:4.2} leads to the existence of $H_{\infty}$ and $H_{\infty}=\lim_{k\to\infty}F(x^k)$.\\
    	\indent Let $(x^{\star},z^{\star},c_{\star})$ be any accumulation point of $\{(x^{k+1},z^k,c_{k+1}):k\in\mathbb{N}\}$ and suppose $\lim_{j \to \infty}(x^{k_j+1},z^{k_j},c_{k_j+1})$ = $(x^{\star},z^{\star},c_{\star})$. First, the continuity of $f/g$ on $\mathcal{X}_{0}$ lead to $c_{\star}=\lim_{j \to \infty}c_{k_j+1}=\lim_{j \to \infty}f/g(x^{k_j+1})=f(x^{\star})/g(x^{\star})$. Secondly, we have $\lim_{j \to \infty}x^{k_j+1}=x^{\star}$ thanks to $\lim_{j \to \infty}x^{k_j}=x^{\star}$ and Corollary \ref{corollary:4.1}. Thirdly, the closedness of subdifferential and $z^{k_j}\in \partial h_2(x^{k_j})$ for $j\in \mathbb{N}$ leads to $z^{\star}\in \partial h_2(x^{\star})$. These together with the continuity of $f$ and $h_1$ as well as the Fenchel-Young equality yield that 
    	\[
    	H(x^{\star},z^{\star},c_{\star})=F(x^{\star})\stackrel{(a)}{=}\underset{j \to \infty}{\lim}F(x^{k_j+1})\stackrel{(b)}{=}\underset{k\to\infty}{\lim}F(x^k),
    	\]
    	where (a) follows from the continuity of $F$ on $\mathcal{X}_{0}$ and (b) holds thanks to Corollary \ref{corollary:4.1}. This completes the proof. 
    \end{proof}
    Finally, we consider the sequential convergence of the solution sequence generated by applying Algorithm \ref{al:2} to problems \eqref{eq:1.2}, \eqref{eq:1.3} and \eqref{eq:1.4}.
    \begin{lemma}\label{lemm5.1}
    For each of problems \eqref{eq:1.2}, \eqref{eq:1.3} and \eqref{eq:1.4}, the function $H$ defined by \eqref{equation:4.2} is a KL function.
    \end{lemma}
    \begin{proof}
    	For problem \eqref{eq:1.2}, one can easily check that $H$ defined by \eqref{equation:4.2} is a semialgebraic function \cite[Page 596]{facchinei2003finite}, and is hence a KL function \cite{attouch2010proximal}. For problem \eqref{eq:1.3}, since the logarithmic function is analytic on $[1,+\infty)$, we have $x\to\|Ax-b\|_{LL_2,\gamma}$ is analytic on $\mathbb{R}^n$. Then, it is obvious that the graph of $H$ in this case is a semianalytic set in $\mathbb{R}^{2n+2}$. This means that $H$ is a subanalytic function\cite[Definition 6.6.1]{facchinei2003finite}. Moreover, $\mathrm{dom}(H)=\mathcal{X}\times\mathbb{R}^n\times\mathbb{R}$ is closed and $H$ is continuous on its domain. Therefore, $H$ is a KL function according to \cite[Theorem 3.1]{bolte2007lojasiewicz}. We finally focus on problem \eqref{eq:1.4}. Noting that $h_2$ in this case is a convex piecewise linear-quadratic function, $h_2^{\star}$ is also piecewise linear-quadratic \cite[Theorem 11.14]{rockafellar1998variational}. Thus, it can be easily checked that $H$ in this case is a semialgebraic function, and hence is a KL function \cite{attouch2010proximal}. The proof is completed.
    \end{proof}
    With the help of Lemma \ref{lemm5.1} and Theorem \ref{theo5.1}, we immediately obtain the following result.
   \begin{theorem} \label{theo:5.2}
   Consider problems \eqref{eq:1.2}, \eqref{eq:1.3} and \eqref{eq:1.4}. Then the solution sequence generated by applying Algorithm \ref{al:2} to each of them converges to a critical point of the respective problem, if the initial point $x^0$ satisfies \eqref{eq4.12} and $\mathcal{X}$ is bounded.
   \end{theorem}
   \section{Numerical experiments}\label{sec:6}
	In this section, we conduct numerical experiments on solving three examples of model \eqref{eq:problem1},  namely, models \eqref{eq:1.2}, \eqref{eq:1.3} and \eqref{eq:1.4} by Algorithm \ref{al:2}. All the experiments are performed in MATLAB R2023b on a desktop with an Intel Core i7-13650HX CPU (2.60GHz) and 24GB of RAM.\\	
	\indent In our experiments, we mainly compare the performance of Algorithm \ref{al:2} with proximal DCA with monotone line search (PDCAL) \cite{gotoh2018dc} for solving the respect model with $\mathcal{X}=\mathbb{R}^n$. Although PDCAL is originally designed for solving DC (Difference of Convex) programs, it can be trivially extended to model \eqref{eq:problem1} with $\mathcal{X}=\mathbb{R}^n$ by using the proximal operator of $\lambda\|\cdot\|_1^2/\|\cdot\|_2^2$. The resulting algorithm is presented in Algorithm \ref{al:3}, where the proximal operator $\text{prox}_{\alpha\lambda\|\cdot\|_1^2/\|\cdot\|_2^2}$ is computed according to \cite[Theorem 3.8]{jia2024computingproximityoperatorsscale}.
     \begin{algorithm}[H]
    	\caption{PDCAL for solving problem \eqref{eq:problem1} with $\mathcal{X}=\mathbb{R}^n$}
    	\label{al:3}
    	\begin{algorithmic}[H]
    		\State \textbf{Step 0.} Input: $x^0 \in \mathbb{R}^n$, $0 < \underline{\alpha} < \overline{\alpha}$, $\sigma > 0$, $0 < r < 1$, and set $k \leftarrow 0$.
    		
    		\State \textbf{Step 1.} Choose $z^k \in A^T\partial q_2(Ax^k-b)$.
    		\State \hspace{3.7em} Set $\alpha:=\widetilde{\alpha}_k\in [\underline{\alpha}, \overline{\alpha}]$.
    		
    		\State \textbf{Step 2.} Compute
    		\[
    		\widehat{x}^{k} \in \mathrm{prox}_{\alpha \lambda\|\cdot\|_1^2/\|\cdot\|_2^2}( x^{k} - \alpha A^T\nabla q_1(Ax^k-b)+\alpha z^k).
    		\]
    		\State \hspace{3.7em} If $\widehat{x}^{k}$  satisfies
    		\[
    		\lambda\frac{\|\widehat{x}^k\|_1^2}{\|\widehat{x}^k\|_2^2} +q(A\widehat{x}^k-b)+\frac{\sigma}{2}\|\widehat{x}^k-x^k\|_2^2
    		\leq \lambda\frac{\|x^k\|_1^2}{\|x^k\|_2^2}+q(Ax^k-b);
    		\]
    		\State \hspace{3.7em} Then, set $x^{k+1}=\widehat{x}^{k}$ and go to \textbf{Step 3};
    		\State \hspace{3.7em} Else, set $\alpha:=\alpha r$ and repeat \textbf{Step 2}.
    		\State \textbf{Step 3.} Record $\alpha_k:=\alpha$, set $k \leftarrow k+1$, and go to \textbf{Step 1}.
    	\end{algorithmic}
    \end{algorithm}
	The experimental setup and test instances are directly adopted from \cite[Section 7]{zeng2021analysis}. For PDCAL and Algorithm \ref{al:2}, we use the same parameters as follows. We set $\underline{\alpha} = 10^{-4}$, $\overline{\alpha}=10^4$, $\sigma = 10^{-3}$ and $r=0.5$ throughout the tests. Inspired by the Barzilai-Borwein method \cite{10.1093/imanum/8.1.141}, we set the initial stepsize $\widetilde{\alpha}_k$ in Step 1 of Algorithm \ref{al:2} as 
	\[
	  \widetilde{\alpha}_k:= \text{max}\left\{\underline{\alpha},\text{min}\left\{\bar{\alpha},\frac{\|\Delta x^k\|_2^2}{|\left\langle \Delta x^k,\Delta q_1^k\right\rangle|}\right\}\right\},
	\]
	where $\Delta x^k = x^k- x^{k-1}$ and $\Delta q_1^k = A^{T}\left(\nabla q_1(Ax^k-b) - \nabla q_1(Ax^{k-1}-b)\right)$ if $k\geq 1$ and $|\left\langle \Delta x^k,\Delta q_1^k \right\rangle|\neq 0$; otherwise, we choose $\widetilde{\alpha}_k =1$. Both PDCAL and Algorithm \ref{al:2} are initialized at the same point as in \cite[Section 7]{zeng2021analysis}, and terminated when 
	\begin{equation}\label{eq:6.1}
		\|x^k-x^{k-1}\|_2\leq \textit{tol} \cdot \text{max}\{\|x^k\|,1\}.
	\end{equation}
    We will specify the choices of initial points and tol in each of the subsection below. In particular, we verify that for all the tests these initial points satisfy \eqref{eq4.12}. Therefore, in view of Theorem \ref{theo:5.2}, we can guarantee the sequential convergence of the solution sequence generated by Algorithm \ref{al:2} and PDCAL in the experiments if they are bounded.
    \subsection{Robust compressed sensing problems \eqref{eq:1.4}.}
    We construct a sensing matrix $A\in\mathbb{R}^{(p+d)\times n}$ with the entries following i.i.d. standard Gaussian distribution and then normalize each column of $A$. Next, the original signal $\widetilde{x}\in\mathbb{R}^n$ is generated as a $K$-sparse vector with $K$ i.i.d. standard Gaussian entries at uniformly chosen random positions. Also, we generate a vector $z\in\mathbb{R}^{p+d}$ with the first $p$ entries being zero and the last $d$ entries being 2sign$(z_{d})$, where the entries of $z_{d}$ are i.i.d. and drawn from standard Gaussian distribution. Finally, the noisy measurement $b\in\mathbb{R}^{p+d}$ is produced as $b=A \widetilde{x} - z +0.01\varepsilon$, where $\varepsilon\in\mathbb{R}^{p+d}$ has i.i.d. standard Gaussian entries. We set the parameters $r=2d$ and $\lambda = 0.01$ for model \eqref{eq:1.4} throughout the tests. We choose $A^{\dagger}b$ as the initial point and set $tol=10^{-6}$ in \eqref{eq:6.1}.\\
    \indent In our tests, we consider $(n, p, K, d)=(2560i,720i,80i,10i)$ with $i\in\{2, 4, 6, 8, 10\}$. For each $i$, 20 random instances are generated as described above. Table 1 summarizes the computational results. We report the computational time in seconds (Time), the objective value (Obj) and the recovery error (RecErr = $\|\widehat{x}-\widetilde{x}\|_2/\text{max}\{1,\|\widetilde{x}\|_2\}$ with $\widehat{x}$ being the output of the algorithms) averaged over 20 random instances. We observe that the objective values and recovery error obtained by Algorithm \ref{al:2} are comparable to those of PDCAL, while Algorithm \ref{al:2} significantly outperforms PDCAL in terms of computational time.
    \begin{table}[htbp]
    	\centering
    	\caption{Random tests on Robust compressed sensing.}
    	\begin{tabular}{ccccccc}
    		\toprule
    		\multirow{2}{*}{$i$} & 
    		\multicolumn{2}{c}{Time(s)} &
    		\multicolumn{2}{c}{Obj} & 
    		\multicolumn{2}{c}{RecErr}\\
    		\cmidrule(lr){2-3} \cmidrule(lr){4-5} \cmidrule(lr){6-7}
    		& PDCAL & Algorithm~\ref{al:2} & PDCAL & Algorithm~\ref{al:2} & PDCAL & Algorithm~\ref{al:2} \\
    		\midrule
    		2 & 1.17 & 0.17 & 1.00e+02 &1.00e+02 & 2.25e-02 & 2.27e-02 \\
    		4 & 4.07 & 0.69 & 2.01e+02 &2.01e+02 & 2.25e-02 & 2.26e-02 \\
    		6 & 8.97 & 1.61 & 3.01e+02 &3.01e+02 & 2.24e-02 & 2.25e-02 \\
    		8 & 16.35 & 2.81 & 4.08e+02 &4.08e+02 & 2.23e-02 & 2.23e-02 \\
    		10 & 24.68 & 4.31 & 5.05e+02 &5.05e+02 & 2.25e-02 & 2.25e-02 \\
    		\bottomrule
    	\end{tabular}
    \end{table}
    \subsection{CS problems with Cauchy noise \eqref{eq:1.3}.}
    As in the previous subsection, we construct a sensing matrix $A\in\mathbb{R}^{m\times n}$ whose entries follow the standard Gaussian distribution and then normalize each column of $A$. The true signal $\widetilde{x}\in\mathbb{R}^n$ is generated as a $K$-sparse vector with $K$ i.i.d. standard Gaussian entries at uniformly chosen random positions. The noisy measurement $b\in\mathbb{R}^m$ is produced as $b=A\widetilde{x} + 0.01\varepsilon$ where $\varepsilon \in \mathbb{R}^m$ has i.i.d. entries drawn from standard Cauchy distribution, i.e., $\varepsilon_{i} = \text{tan}(\pi(\widetilde{\varepsilon}_{i}-\frac{1}{2}))$ for $i= 1,2,\cdots,m$ with $\widetilde{\varepsilon}_{i}\in\mathbb{R}^m$ has i.i.d. entries uniformly chosen in $[0,1]$. We set  $\gamma=0.02$ and $\lambda=40$ for model \eqref{eq:1.3} throughout the tests.\\
    \indent The initial points are computed by applying the sequential convex programming (SCP) method \cite{yu2021convergencerateanalysissequential} to the following $\ell_1$ minimization model 
    \[
      \underset{x\in\mathbb{R}^n}{\text{min}} \left\{\|x\|_1:\|Ax-b\|_{LL_2,\gamma}\leq\eta \right\},
    \]
    where $\eta$ is chosen as $1.2\|0.01\varepsilon\|_2$. We use the same parameters for SCP method as in \cite[section 5]{yu2021convergencerateanalysissequential}, except that SCP is terminated when \eqref{eq:6.1} is satisfied with $tol=10^{-6}$. Also, for both PDCAL and Algorithm \ref{al:2}, we set $tol=10^{-6}$ in \eqref{eq:6.1}.\\
    \indent In our experiments, we consider $(n,m,K)=(2060i,720i,80i)$ with $i\in \left\{2,4,6,8,10\right\}$. For each $i$, 20 random instances are generated as described above. The computational results are presented in Table 2, where we report the computational time in seconds (Time), the objective values (Obj) and the recovery error (RecErr) averaged over 20 random instances. We observe that Algorithm \ref{al:2} can achieve the same objective values and recovery accuracy at about one fifth of the computational time.
    
    \begin{table}[htbp]
    	\centering
    	\caption{Random tests on CS problems with Cauchy noise.} %(tol = $10^{-6}$ for $SCP_{ls}$)}
    	\begin{tabular}{ccccccc}
    		\toprule
    		\multirow{2}{*}{$i$} & 
    		\multicolumn{2}{c}{Time(s)} & 
    		\multicolumn{2}{c}{Obj} & 
    		\multicolumn{2}{c}{RecErr}\\
    		\cmidrule(lr){2-3} \cmidrule(lr){4-5} \cmidrule(lr){6-7}
    		& PDCAL & Algorithm~\ref{al:2} & PDCAL & Algorithm~\ref{al:2} & PDCAL & Algorithm~\ref{al:2} \\
    		\midrule
    		2 & 0.98 &  0.15 & 1.01e+02 &1.01e+02 & 4.13e-02 & 4.13e-02\\
    		4 & 3.71 &  0.64 & 1.99e+02 &1.99e+02 &4.26e-02 & 4.26e-02\\
    		6 & 8.14 &  1.61 & 3.02e+02 &3.02e+02 & 4.19e-02 & 4.19e-02\\
    		8 &  13.89 &  2.59 & 4.06e+02 &4.06e+02 &4.13e-02 & 4.13e-02\\
    		10 & 21.87 &  4.15 & 5.05e+02 &5.05e+02 &4.14e-02 & 4.14e-02\\
    		\bottomrule
    	\end{tabular}
    \end{table}
	\subsection{CS problems with Gaussian noise \eqref{eq:1.2}.}
	In this section, we conduct a sensing matrix $A \in \mathbb{R}^{m\times n}$ by oversampled discrete cosine transformation, i.e., $A=[a_1,a_2,\cdots,a_n]\in\mathbb{R}^{m\times n}$ with
	\[
	  a_j = \frac{1}{\sqrt m} \text{cos} \left(\frac{2\pi\omega j}{F}\right),j=1,2,\cdots,n,
	\]
    where $\omega\in\mathbb{R}^m$ is a random vector following the uniform distribution in $[0,1]^m$ and $F>0$ is a parameter measuring how coherent the matrix is. Then the original signal $\widetilde{x} \in \mathbb{R}^n$ is generated as a $K$-sparse signal with a dynamic range $D>0$. In detail, we generate $\widetilde{x}$ via the MATLAB command below:
    \begin{center}
    \begin{verbatim}
               I = randperm(n); J=I(1:K); x_tilde = zeros (n,1)
               x_tilde(J) = sign(randn(K,1)).*10^(D*rand(K,1));
    \end{verbatim}
    \end{center}
	Next, we produce the noisy measurement $b \in \mathbb{R}^m$ by $b= A\widetilde{x}+0.01\varepsilon$, where $\varepsilon \in \mathbb{R}^m$ has i.i.d. a standard Gaussian entries. We set $\lambda = 0.4$ for model \eqref{eq:1.2} throughout the tests.\\
	\indent As in \cite[Section 7.3]{zeng2021analysis}, we obtain the initial points by applying SPGL method \cite{doi:10.1137/080714488} with default settings to the following $l_1$ minimization model
	\[
	 \underset{x \in \mathbb{R}^n}{\text{min}} \left\{\|x\|_1:\|Ax-b\| \leq \eta \right\},
	\]
	where $\eta$ is chosen as $1.2\|0.01\varepsilon\|_2$. Also, we set $tol = 10^{-8}$ for the compared algorithms.\\
	\indent In our tests, we fix $n=1024$, $m=64$, and consider $K\in\left\{8,12\right\}$, $F\in\left\{5,15\right\}$, $D\in\left\{2,3\right\}$. For each triple $(K,F,D)$, 20 random instances are generated as described above. The computational results are shown in Table 3, where we present the computational time in seconds (Time), the objective values (Obj) and recovery error (RecErr) are averaged over 20 random instances. We see that comparable objective values and recovery error are achieved by Algorithm \ref{al:2}  at about half the computational time of PDCAL.    

	\begin{table}[htbp]
		\centering
		\caption{Random tests on badly scaled CS problems with Gaussian noise.}
		\begin{tabular}{ccccccccc}
			\toprule
			\multirow{2}{*}{$k$} &
			\multirow{2}{*}{$F$} &
			\multirow{2}{*}{$D$} & 
			\multicolumn{2}{c}{Time(s)} &
			\multicolumn{2}{c}{Obj} &  
			\multicolumn{2}{c}{RecErr}\\
			\cmidrule(lr){4-5} \cmidrule(lr){6-7}
			\cmidrule(lr){8-9}
			& & & PDCAL & Algorithm~\ref{al:2} & PDCAL & Algorithm~\ref{al:2} & PDCAL & Algorithm~\ref{al:2}\\
			\midrule
			8 & 5 & 2 & 0.05 & 0.02 & 4.04e+00 & 4.04e+00 & 3.888e-03 &  3.888e-03\\
			8 & 5 & 3 & 0.34 & 0.17 & 2.74e+00 & 2.74e+00 & 6.193e-04 &  6.195e-04 \\
			8 & 15 & 2 & 1.05 & 0.49 & 3.99e+00 & 3.99e+00 &  1.619e-01 & 1.619e-01 \\
			8 & 15 & 3 & 47.36 & 20.91 & 2.85e+00 & 2.85e+00  & 5.433e-02 & 5.432e-02 \\
			12 & 5 & 2 & 0.55 & 0.27 & 5.15e+00 & 5.15e+00 &3.634e-02 & 3.634e-02 \\
			12 & 5 & 3 & 8.88 & 4.13 & 3.80e+00 & 3.80e+00 & 3.682e-03 &  3.682e-03 \\
			12 & 15 & 2 & 1.90 & 0.93 & 5.53e+00 & 5.53e+00 & 1.592e-01 & 1.592e-01 \\
			12 & 15 & 3 & 210.38 & 98.21 & 4.18e+00 & 4.18e+00 & 7.891e-01 & 7.936e-01  \\
			\bottomrule
		\end{tabular}
	\end{table}

	% \appendix
	% \newpage 
	% \bibliographystyle{plain}
	\bibliographystyle{plain}
	\bibliography{document2re2}
\end{document}